\documentclass[a4paper]{siamltex1213}

\usepackage{moreverb}
\usepackage{multirow}
\usepackage{enumerate}
\usepackage{graphicx, caption, subcaption}
\usepackage{multirow, array}
\usepackage{amsmath, amsfonts, amssymb, xfrac}
\usepackage{tikz}
\usetikzlibrary{shapes, snakes, patterns}
\usetikzlibrary{calc}
\usepackage{pgfplots}
\usepackage{standalone}
\usepackage{graphicx}
\usepackage{graphics}
\usepackage{color}
\usepackage{manfnt}
\usepackage{mathtools}

\usepackage{ifpdf}
\ifpdf
\fi

\newtheorem{example}{Example}
\numberwithin{example}{section}

\newtheorem{remark}{Remark}
\numberwithin{remark}{section}

\DeclarePairedDelimiter\floor{\lfloor}{\rfloor}
\newcommand*{\Scale}[2][4]{\scalebox{#1}{$#2$}}%
\newcommand{\reals}{\mathbb{R}}

\newcommand{\set}[1]{\{#1\}}
\newcommand{\abs}[1]{\lvert#1\rvert}
\newcommand{\semi}[1]{\lvert#1\rvert}
\newcommand{\norm}[1]{\lVert#1\rVert}
\newcommand{\nhalf}{\Scale[0.5]{-\tfrac{1}{2}}}
\newcommand{\dual}[1]{#1^{\prime}}
\newcommand{\Div}[1]{\nabla\cdot #1}
\renewcommand{\brack}[1]{\langle#1\rangle}

\newcommand{\Zpolar}{Z^0}
\newcommand{\Zperp}{Z^\perp}
\newcommand{\inv}[1]{\ensuremath{{#1}^{-1}}}
\newcommand{\tr}[1]{\ensuremath{{\text{tr}{#1}}}}
\newcommand{\transp}[1]{\ensuremath{{#1}^{\Scale[0.5]{\top}}}}
\newcommand{\la}[1]{\mathbf{#1}}
\DeclareMathOperator{\spn}{span}

\makeatletter
\renewcommand{\inf}{\mathop{\@inf\vphantom{\@sup}}}
\renewcommand{\sup}{\mathop{\@sup\vphantom{\@inf}}}
\newcommand{\@inf}{\operatorname*{inf}}
\newcommand{\@sup}{\operatorname*{sup}}

\makeatother

\usepackage[final]{changes}  
\definechangesauthor[name={rev1}, color=red]{rev1}            
\definechangesauthor[name={rev2}, color=green]{rev2}          
\definechangesauthor[name={revboth}, color=magenta]{revboth}  
\setremarkmarkup{(#2)}

\title{On the Singular Neumann Problem in Linear Elasticity
\thanks{
  submitted to Numerical Linear Algebra with Applications.
}
}

\author{
Miroslav Kuchta\footnotemark[2] \and
Kent-Andre Mardal\footnotemark[2]\footnotemark[3] \and
Mikael Mortensen\footnotemark[2]
}

\begin{document}
\renewcommand{\thefootnote}{\fnsymbol{footnote}}
\footnotetext[2]{Department of Mathematics, Division of Mechanics, University of
Oslo; \texttt{\{mirok, kent-and, mikaem\}@math.uio.no}}
\footnotetext[3]{Center for Biomedical Computing, Simula Research Laboratory}
\renewcommand{\thefootnote}{\arabic{footnote}}

\maketitle

\begin{abstract}
The Neumann problem of linear elasticity is singular with a kernel formed by the 
rigid motions of the body. There are several tricks that are commonly used to obtain a non-singular linear system. However, they often cause reduced accuracy or 
lead to poor convergence of the iterative solvers.
In this paper, \deleted{four} different well-posed formulations of the problem are studied 
through discretization by the finite element method, and preconditioning strategies 
based on operator preconditioning are discussed.
For each formulation we derive 
preconditioners that are independent of the discretization parameter. Preconditioners that 
  are robust with respect to the first 
Lam\'{e} constant are constructed for the pure displacement formulations, while 
a preconditioner that is robust in both Lam\'{e} constants is constructed for the 
mixed formulation.
It is shown that, for convergence in the first Sobolev norm, it is crucial to 
respect the orthogonality constraint derived from the continuous problem. Based on 
this observation a modification to the conjugate gradient method is proposed that
achieves optimal error convergence of the computed solution.
\end{abstract}


\pagestyle{myheadings}
\thispagestyle{plain}
\markboth{On the Singular Neumann Problem in Linear Elasticity}{}

\vspace{-6pt}

\section{Introduction}\label{sec:intro}
The presented paper discusses numerical techniques for solving the singular
problem of linear elasticity. Let $\Omega\subset\mathbb{R}^3$ be the body subjected 
to volume forces $f:\Omega\rightarrow\reals^3$ and surface forces
$h:\partial\Omega\rightarrow\reals^3$. The body's displacement $u:\Omega\rightarrow\reals^3$ 
is then found as a solution to
\begin{equation}
  \label{eq:strong}
  \begin{aligned}
-&\nabla\cdot\sigma(u) = f                                      &\mbox{ in }\Omega,\\
 &\sigma(u) = 2\mu\epsilon(u) + \lambda(\nabla\cdot u){I}&\mbox{ in }\Omega,\\
 &\sigma(u)\cdot n = h                                         &\mbox{ on }\Gamma_N = \partial\Omega,\\
  \end{aligned}
\end{equation}
with $\mu>0$, $\lambda\geq 0$ the Lam\'{e} constants of the material, ${I}$ the 
identity matrix, $\epsilon(u)=\tfrac{1}{2}\left(\nabla u + \transp{\left(\nabla u\right)}\right)$ 
the strain and $n$ the outward-pointing surface normal, see \cite{marsden}.
We note that the constitutive law for the stress tensor $\sigma$ can be equivalently stated
as $\sigma(u)=2\mu\epsilon(u) + \added[id=rev2, remark={comment 1}]{\lambda}\tr{(\epsilon(u))}{I}$ where $\tr{(\epsilon(u))}$
denotes the trace of $\epsilon(u)$, i.e. the sum of its diagonal.

%
%
The system is used extensively in structural analysis \cite{Bauchau2009}, and is 
relevant in numerous applications e.g., marine engineering \cite{DNV}, 
biomechanics of brain \cite{brain}, spine \cite{spine} or the mechanics of 
planetary bodies \cite{OC}.

Due to the absence of a Dirichlet boundary condition that can anchor the body (coordinate system) 
in space, the solution can be uniquely determined if and only 
if the net force and the net torque on $\Omega$ are zero, i.e., the forces $f$, $h$ satisfy the 
compatibility conditions
\begin{equation}
  \label{eq:compat}
  \begin{aligned}
    &\int_{\Omega} f\,\mathrm{d}x + \int_{\partial\Omega} h\,\mathrm{d}s &= 0,\\
    &\int_{\Omega} f\times x \,\mathrm{d}x + \int_{\partial\Omega} h\times x \,\mathrm{d}s &= 0.
  \end{aligned}
\end{equation}
With such compatible data the now solvable \eqref{eq:strong} is singular as any 
rigid motion can be added to the solution. We note that the space of rigid
motions $z:\Omega\rightarrow\reals^3$ such that $\epsilon(z)=0$, consists of 
translations and rigid rotations and for a body in $3d$ the space is 
six-dimensional.

The ambiguity of the solution of \eqref{eq:strong} can be removed by adding
constraints by means of Lagrange multipliers which enforce that the solution is
free of rigid motions. When discretized, this approach yields an invertible 
saddle point system. Alternatively, discretizing \eqref{eq:strong} directly
leads to a symmetric, positive semi-definite matrix with a six dimensional kernel. 
Singular systems may be solved by iterative methods if care is taken 
to handle the kernel during the iterations, but  a common approach (here termed \textit{pinpointing}) 
in engineering literature, e.g. \cite{DNV}, is to remove the nullspace by 
prescribing the displacement in selected points of $\partial{\Omega}$. 

\added[id=rev1, remark={comment2}]{
If $\Gamma_N\neq \partial \Omega$ and a Dirichlet boundary condition is
prescribed on $\partial\Omega\setminus\Gamma_N$ the equations of linear elasticity
are well posed (e.g. \cite[ch. 6.3]{braess}) and there exists a number of
efficient solution algorithms for the problem. Here we discuss some of the methods
for which the Neumann problem \eqref{eq:strong}, or more precisely, correct treatment
of the rigid motions, is relevant.}

\added[id=rev1, remark={comment2}]{
In the context of algebraic multigrid (AMG) it is recognized already in
the early work of Ruge and St{\" u}ben \cite{ruge1987algebraic} that carefully
constructed interpolators are need to obtain good convergence for problems
stemming from equations of linear elasticity (PDE systems in general). In particular,
the authors observe that with the so called ``unknown'' approach convergence of AMG
deteriorates when then number of Dirichlet boundaries decreases. The issue here is
that with the ``unknown'' approach only the translations are interpolated well
on the coarse grid, cf. \cite{baker2010improving}, and as a remedy the authors
propose to improve the interpolation of rotations (eigenvectors with small
eigenvalues in general). Griebel et al. \cite{griebel2003algebraic} construct a
block-interpolation where the rotations are captured exactly if the underlying
grid is point-symmetric. However, this assumption fails to hold at the boundary
nodes and AMG becomes less effective as the number of Neumann boundaries
increases. More recently \cite{baker2010improving} discusses computationally efficient
techniques for augmenting a given/existing AMG interpolator to ensure exact
interpolation of rigid motions (nullspace vectors in general). A related
approach is \cite{vassilevski2006multiple} who derive algorithms for constructing
AMG interpolation operators which exactly interpolate any given set of
vectors. The requirement that the coarse space captures rigid motions is also
found in the later variants of AMG. For example, in smoothed aggregation AMG
\cite{vanvek1996algebraic, mandel1999energy} the coarse basis functions are
constructed from a (global) constrained minimization problem where preservation of
the nullspace is one of the constraints. The minimization problems solved
in construction of AMG based on element interpolation \cite{brezina2001algebraic,
  jones2001amge, henson2001element} uses rigid motions of the local stiffness
matrices. Similarly, the kernel of local stiffness matrices is preserved 
by the approximate splittings in AMG based on computational
molecules \cite{kraus2008algebraic, karer2010algebraic}. To complete our
(non-exhaustive) list let us mention that the in the domain decomposition
methods, e.g. FETI \cite{farhat1991method}, the Neumann problem \eqref{eq:strong}
arises naturally on ``floating'' subdomains that do not intersect the Dirichlet
boundaries. Here, the local singular problem is treated algebraically by
pseudoinverse (cf. the discussion in \S \ref{sec:krylov}).}

\added[id=rev1, remark={comment1}]{
  In the following we aim to solve \eqref{eq:strong} with the finite element method (FEM)
  while using Krylov methods for the resulting linear systems. As the systems
  are singular the Krylov solvers are initialized with the nullspace of
  rigid motions (typically in the form of the $l^2$ orthonormal set of vectors).
  In the standard implementation} \footnote{See e.g.\\
  \url{http://www.mcs.anl.gov/petsc/petsc-current/docs/manualpages/KSP/KSPSolve.html}}
\added[id=rev1, remark={comment1}]{
  the Krylov methods employ the same ($l^2$) projection to orthogonalize both
  the right hand side as well as the solution vector with respect to the given
  nullspace. A particular question that we address here is then whether these
  algorithms provide discrete approximations which converge to the weak
  solution of \eqref{eq:strong} in the $H^1$ norm. We shall see that, in general,
  the answer is negative and that the issue stems from the fact that in the context
  of FEM a vector in $\reals^n$ can be associated with a \emph{function} from the
  finite dimensional finite element space $V_h\subset H^1$, i.e. it represents a
  solution/left hand side, as well as with the \emph{functional} from the corresponding dual space,
  that is, it is a representation of the right hand side. Consequently two
  projectors are required in iterative method originating from a singular variational
  problem. However, standard implementations of Krylov methods, which employ
  single projection, fail to make the distinction.}


\added[id=rev1, remark={comment1}]{
%
Rewriting the Krylov solvers to take the two representations into account is in principle a simple addition
to the code. However, it is also intrusive and to the best of our knowledge this distinction is 
not implemented in state-of-the-art linear algebra frameworks such as 
PETSc\cite{petsc} or Hypre\cite{hypre}. Here, we therefore propose a simple
alternative solution which is less intrusive.
To this end, we focus on analysis of the Lagrange multiplier method
and the conjugate gradient (CG) method for the singular problem \eqref{eq:strong}.
Well-posedness of both the methods is discussed and robust preconditioners are 
established based on operator preconditioning \cite{kent}. Further, connections 
between the two methods and the question of whether they yield identically 
converging numerical solutions are elucidated. These methods rely on standard iterative
solvers as they implicitly contain the two required projectors. 
}

The manuscript is structured as follows. In \S \ref{sec:prelim} the necessary notation
is introduced and shortcomings of pinpointing and CG are illustrated by
numerical examples. Section \ref{sec:lagrange} discusses Lagrange multiplier
formulation and two preconditioners for the method. Section \ref{sec:krylov} deals 
with the preconditioned CG method and two preconditioners are proposed. Further, 
it is revealed that if the continuous origin of the discrete problem is
ignored, the method, in general, will not yield convergent solutions. A continuous variational 
setting is introduced to modify the CG to yield a convergent method. Section \S \ref{sec:energy} 
discusses well-posedness and preconditioning of an alternative formulation of 
\eqref{eq:strong}. The proposed formulation leads to a symmetric, positive
definite linear system. In \S\ref{sec:lagrange}-\S\ref{sec:energy} we assume that 
$\lambda$ and $\mu$ are of comparable magnitude in order to put the focus on proper 
handling of the rigid motions.  In \S \ref{sec:mixed} we consider the case where 
$\lambda \gg \mu$. The focus here is on a well-known and simple technique to remove 
the problems of locking, namely the mixed formulation of linear elasticity where
an extra unknown, the \textit{solid pressure} is introduced. \added{We discuss
  two formulations which yield }robust approximation and preconditioning in $\lambda$ when 
care is taken of proper handling of the rigid motions. Finally, conclusions are drawn 
in \S \ref{sec:fin}.

\section{Preliminaries}\label{sec:prelim}
Let $V$ be the Sobolev space of 
vector (or scalar or tensor) valued functions, which, together with their weak derivatives
of order one, are in space $L^2(\Omega)$. We denote by 
$(\cdot, \cdot)$ the $L^2(\Omega)$ inner product of functions in $V$ while $\norm{\cdot}$ is the 
corresponding norm.
%
For the $L^2$ inner product over boundary $\partial\Omega$ we shall use the notation
$(\cdot, \cdot)_{\partial\Omega}$.
The standard inner product of $V$ is $(u, v)_1=(u, v)+(\nabla u, \nabla v)$, $u, v\in V$ and 
$\norm{\cdot}_1$ shall be the induced norm. For any Hilbert space $V$ its dual 
space is denoted as $\dual{V}$ and we use capital or calligraphy letters to 
denote operators, e.g. $A:V\rightarrow\dual{V}$ or $\mathcal{A}:(V\times
V)\rightarrow\dual{(V\times V)}$. Finally, $\brack{\cdot, \cdot}$ is the duality 
pairing between $\dual{V}$ and $V$.

The space $\reals^n$ is considered with the $l^2$ inner product $\transp{{x}}{y}=x_iy_i$ 
(invoking the summation convention), ${x}, {y}\in\reals^n$ and the norm 
$\semi{{x}}=\sqrt{\transp{{x}}{x}}$. For clarity of notation bold fonts 
are used to denote vectors and operators(matrices) in $\reals^n$ that are
representations of functions and operators from finite dimensional finite element approximation 
space $V_h\subset V$. Let $\set{\phi_i}_{i=1}^{n}$ be the nodal basis of $V_h$. The 
representations are obtained by mappings $\pi_h:V_h\rightarrow\reals^n$ (the nodal interpolant) 
and $\mu_h:\dual{V}_h\rightarrow\reals^n$ such that for $v\in V_h$, $f\in\dual{V}_h$
\begin{equation}\label{eq:pimu}
  v=(\pi_h v)_i\phi_i
  \quad\mbox{ and }\quad
  (\mu_h f)_i = \brack{f, \phi_i}.
\end{equation}
We refer to \cite[ch 6.]{kent} for a detailed discussion of the properties of the
mappings, e.g. invertibility, and note here that $M:V_h\rightarrow \dual{V_h}$ 
is represented by a matrix $\la{M}=\mu_h M \inv{\pi_h}$. In particular, the mass
matrix $\la{M}$, $M_{ij}=(\phi_j, \phi_i)$ represents the Riesz map with respect to 
the $L^2$-inner product, $\brack{Mu, v}=(u, v)$, $u \in V_h$. 
On the other
hand the duality pairing between $\dual{V_h}$ and $V_h$ is represented by the $l^2$
inner product $\brack{f, v}=\transp{\la{f}}\la{v}$, $\la{f}=\mu_h f$, $\la{v}=\pi_h v$. 
We remark that for $V_h$ set up on a sequence of non-uniformly refined triangulations 
of $\Omega$, the $l^2$ inner product $\transp{\la{u}}\la{v}$ where $\la{v}=\pi_h v$, 
$\la{u}=\pi_h u$ may not provide a converging approximation of $(u, v)$ and the 
distinction between the two becomes crucial for the construction of converging methods.

Finally, Korn's inequalities on $V=\left[ H^1(\Omega) \right]^3$ and 
$\Zperp=\set{v\in V; (v, z)=0\,\forall z\in Z}$, $Z=\set{v\in V; \epsilon(v)=0}$ are 
invoked, see \cite[thm 2.1]{ciarlet} and \cite[thm 2.3]{ciarlet}. There exist a positive 
constant $C=C(\Omega)$ such that 
\begin{equation}\label{eq:kornV}
  C\norm{u}^2_1\leq\norm{\epsilon(u)}^2+\norm{u}^2 \quad \forall u\in V.
\end{equation}
There exists a positive constant $C=C(\Omega)$ such that
\begin{equation}\label{eq:kornZ}
  C\norm{u}^2_1 \leq \norm{\epsilon(u)}^2\quad \forall u\in\Zperp.
\end{equation}

\replaced{To motivate out investigations and illustrate the lack of $H^1$
  convergence that pinpointing or standard CG can lead to, we present three
numerical examples.}{To motivate our investigations, we present three numerical examples which
discuss performance of CG and pinpointing for solving \eqref{eq:strong}}. That
the pinpointing can be a suitable method for treating a singular problem is 
shown in the first example which considers the Poisson problem with Neumann boundary 
conditions. However, pinpointing does not work well with \eqref{eq:strong} as
the second example shows. In the third example, the 
singular elasticity problem is finally solved with preconditioned CG. \deleted{The 
employed preconditioner ignores the rigid motions leading to lack of convergence 
and unbounded iterations.}


Bochev and Lehoucq \cite{bochev} report an increase in iteration count due to 
pinpointing for a CG method without a preconditioner in the context of singular Poisson
problem. However, Krylov methods are in practice rarely applied without a preconditioner. 
For this reason, Example \ref{ex:pin_poisson} solves the singular Poisson problem 
in two and three dimensions by means of pinpointing and a preconditioned CG.
\deleted{We will see that the preconditioned method yields convergent numerical solutions without 
increasing the iteration count.}
\begin{example}\label{ex:pin_poisson} We consider $\Omega=\left[0, 1\right]^d$, 
$d=2, 3$ and the singular Poisson equation
\[
  \begin{aligned}
    -&\Delta u = f         &\mbox{ in }\Omega,\\
     &\nabla u\cdot n = 0  &\mbox{ on }\partial\Omega,
  \end{aligned}
\]
with unique exact solution obtained by subtracting its 
mean value $\inv{\abs{\Omega}}\int_{\Omega} u\,\mathrm{d}x$ from a manufactured $u$. 
The value of the exact solution is prescribed as a constraint for the degree of 
  freedom at the (bottom) lower left corner of the domain, which is triangulated 
  such that the computational mesh is refined towards the origin.

  To discretize the system continuous linear Lagrange
  elements\footnote{Unless stated otherwise continuous linear Lagrange elements
  ($P_1$) are used to discretize all the presented numerical examples.} from the FEniCS 
library \cite{fenics, fenics_old} were used. The resulting linear system was solved by
the preconditioned CG method implemented in the PETSc library \cite{petsc}, using 
HypreAMG \cite{hypre} to compute the action of the preconditioner.
\added[id=rev1, remark={comment 4}]{More specifically we used a single $V$
  cycle with one pre and post smoothing by a symmetric-SOR smoother. The
  other AMG parameters were kept at their default settings, e.g. \emph{classical} interpolation,
  \emph{Falgout} coarsening. 
}
\footnote{The settings for AMG were reused throughout all the
    numerical experiments presented in the paper.
  }
The iterations were started from a random initial guess and a relative preconditioned
residual magnitude of $10^{-11}$ was required for convergence.

The number of iterations together with error and convergence rate based on the $H^1$ 
norm are reported in Table \ref{tab:poisson}. Pinpointing yields numerical solutions 
$u_h$ that converge at optimal rate.
%
Moreover, the number of iterations is
  bounded. Unlike in \cite{bochev} where specifying the solution datum in single point 
  was found to lead to increasing number of \textit{unpreconditioned} CG iterations 
  (both in 2$d$ and 3$d$) we find here that \textit{preconditioned} CG with the system 
  modified by pinpointing is a suitable numerical method for the singular Poisson
  problem.

\begin{table}[ht!]
  \begin{center}
  \caption{Convergence of the pinpointing approach for the singular Poisson
    problem.}
\footnotesize{
\begin{tabular}{llc|llc}
\hline
\multicolumn{3}{c|}{$d=2$} & \multicolumn{3}{c}{$d=3$}\\
\hline
 size & $\norm{u-u_h}_1$ & \# & size & $\norm{u-u_h}_1$ & \# \\
\hline
  40849   & 2.49E-01 (1.00) & 11 & 12347   & 2.72E+00 (1.22) & 10\\
  162593  & 1.25E-01 (1.00) & 11 & 92685   & 1.36E+00 (1.01) & 11\\
  648769  & 6.23E-02 (1.00) & 11 & 718649  & 6.78E-01 (1.00) & 12\\
  2591873 & 3.11E-02 (1.00) & 12 & 5660913 & 3.39E-01 (1.00) & 12\\
\hline
\end{tabular}
}
\label{tab:poisson}
\end{center}
\end{table}
\end{example}

Following the performance of pinpointing in the singular Poisson problem, the same 
approach is now applied to \eqref{eq:strong} in Example \ref{ex:pin_elasticity}.
\deleted{Here, we will observe that fixing the solution datum in vertices of the mesh leads 
to slightly increased iteration counts. More importantly, we will see that the method in
general does not yield converging solutions.}
\begin{example}\label{ex:pin_elasticity} We consider the singular elasticity problem
  \eqref{eq:strong} with $\mu=384$, $\lambda=577$ and $\Omega$ obtained by rigid 
  deformation of the box $\left[-\tfrac{1}{4}, \tfrac{1}{4}\right]\times \left[-\tfrac{1}{2}, \tfrac{1}{2}\right]\times\left[-\tfrac{1}{8}, \tfrac{1}{8}\right]$. 
The box was first rotated around $x$, $y$ and $z$ axes by angles $\tfrac{\pi}{2}$, $\tfrac{\pi}{4}$ 
and $\tfrac{\pi}{5}$ respectively. Afterwards it was translated by the vector $(0.1, 0.2, 0.3)$.
%
  Starting from 
  $u^{*}=\tfrac{1}{4}(\sin{\tfrac{\pi}{4}x}, z^3, -y)$ the unique solution $u$ of \eqref{eq:strong}
  is constructed by orthogonalizing $u^{*}$ with respect to the rigid motions of $\Omega$, where 
  the orthogonality is enforced in the $L^2$ inner product, while the right hand side $f$ is
  manufactured by adding to $-\nabla\cdot\sigma(u)$ a linear combination of rigid motions.
  Finally, we take $\sigma(u)\cdot n$ as the surface force $h$.
  The solution is pictured in Figure \ref{fig:uexact}. We note that in this example 
  a uniform triangulation is used.

To obtain from \eqref{eq:strong} an invertible linear system, the exact
displacement was prescribed in four different ways, cf. Table
\ref{tab:pin_elasticity} below. ($3\circ$) constrains six degrees of freedom in 
  three corners of the body such that in $i$-th corner there are $i$ components 
  prescribed. This choice is motivated by the dimensionality of the space of
  rigid motions, cf. \cite{DNV}. The fact that fixing three points in space is sufficient to prevent
the body from rigid motions motivates ($1\triangleright$) where all three components of 
displacement are prescribed on vertices of a single triangular element on $\partial\Omega$. 
However, with mesh size decreasing this constraint effectively becomes a
constraint for a single (mid)point. Thus in ($3\triangleright$) the displacement
in three arbitrary triangles is fixed. Finally in ($3\bullet$) the displacement 
is prescribed in three corners of the body. 

The iterative solver used the same tolerances and parameters as in Example \ref{ex:pin_poisson}. 
In particular, identical settings of the multigrid preconditioner were utilized and
the iterations were started from random initial vector. \added{We note that AMG
was \emph{not} initialized with the rigid motions.}

The number of iterations together with error and convergence rates based on the $H^1$ 
norm are reported in Table \ref{tab:pin_elasticity}. Note that all the
considered pinpointing strategies lead to moderately increased iteration counts.
The increase is most notable for ($1\triangleright$), which effectively constrains 
  a single point as the mesh is refined. On the other hand, strategies ($3\triangleright$) 
  and ($3\bullet$), that always constrain all three components of the displacement 
  in at least three points, yield the slowest growth rates. However, neither strategy yields
convergent numerical solutions. In fact, the numerical error can often be seen to 
  increase with resolution. 

\begin{table}[ht!]
  \begin{center}
  \caption{Convergence of the pinpointing approach for the singular elasticity
    problem.}
\footnotesize{
\begin{tabular}{l|lc|lc|lc|lc}
\hline
\multirow{2}{*}{size} & \multicolumn{2}{c|}{$3\circ$}
                      & \multicolumn{2}{c|}{$1\triangleright$}
                      & \multicolumn{2}{c|}{$3\triangleright$} 
                      & \multicolumn{2}{c}{$3\bullet$}\\
  \cline{2-9}
& $\norm{u-u_h}_1$ & \# & $\norm{u-u_h}_1$ & \# & $\norm{u-u_h}_1$ & \# & $\norm{u-u_h}_1$ & \#\\
\hline
2187   & 6.69E-02 (-0.02) & 30 & 1.01E-01 (-0.70) & 32 & 2.82E-02 (0.88)  & 24 & 2.89E-02 (0.99)  & 25\\
14739  & 1.27E-01 (-0.92) & 35 & 9.61E-01 (-3.25) & 40 & 1.08E-02 (1.38)  & 28 & 1.35E-02 (1.10)  & 29\\
107811 & 2.57E-01 (-1.02) & 36 & 7.89E+00 (-3.04) & 48 & 1.72E-02 (-0.66) & 31 & 1.08E-02 (0.31)  & 32\\
823875 & 5.17E-01 (-1.01) & 41 & 6.36E+01 (-3.01) & 54 & 3.96E-02 (-1.21) & 33 & 1.82E-02 (-0.75) & 35\\
\hline
\end{tabular}
}
\label{tab:pin_elasticity}
\end{center}
\end{table}
\end{example}
\begin{figure}
  \begin{center}
  \includegraphics[width=0.5\textwidth]{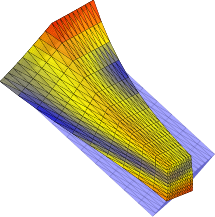}
  \end{center}
  \caption{Computational domain (blue) deformed by exaggerated(4x) analytical 
  displacement used in the numerical examples. The deformed body is colored
  by the magnitude of the displacement.
  }
  \label{fig:uexact}
\end{figure}

In the final example a preconditioned CG method will be applied to solve the 
singular elasticity problem with data such that the compatibility conditions 
\eqref{eq:compat} are met. \deleted{Based on whether or not the components of the kernel are 
removed from the converged vector, we will see that the method yields convergent/divergent 
numerical solutions. We will also see that the iteration counts are not bounded.}

\begin{example}\label{ex:krylov_AM}
    We consider a modified problem from Example \ref{ex:pin_elasticity} where
    $f$ is not perturbed by rigid motions. As the data satisfy \eqref{eq:compat}, 
    the discrete linear system is solvable and amenable to solution by the preconditioned 
    CG method. To this end the rigid motions are passed to the conjugate gradient
    solver via the PETSc interface
    \footnote{See \emph{MatSetNullSpace}\\
    \url{http://www.mcs.anl.gov/petsc/petsc-3.5/docs/manualpages/Mat/MatSetNullSpace.html}
    }.
    The mass or identity matrix is added to the singular system 
    matrix in order to obtain a positive definite matrix in the construction of the
    preconditioner based on AMG. The first choice can be viewed as a simple mean
    to get an invertible system while the motivation for the latter is the
    functional setting to be discussed later in Theorem \ref{thm:h1_norm}. 
  Moreover, for each preconditioner two cases are considered where the converged 
  vector is either postprocessed by removing from it the components of the nullspace or
  no postprocessing is applied. We note that in this example the iterations are started 
  from a zero initial vector and the relative tolerance of $10^{-10}$ is used as a convergence
  criterion.
  The number of iterations together with error and convergence rates based on the $H^1$ 
  norm are reported in Table \ref{tab:krylov_AM}. We observe that the method with
  the mass matrix (cf. left pane of the table) yields convergent solutions only if 
  postprocessing is applied. On the other hand solutions with the preconditioner 
  based on the identity matrix converge in the $H^1$ norm even if no postprocessing is
  used. The observation that the \added{Krylov iterations/}preconditioners respectively do and do not introduce 
  rigid motions (recall that the initial guess and right hand side are orthogonal 
  to the kernel) is related to properties of the added matrices. A vector
  free of rigid motions remains orthogonal after applying to it the identity matrix.
  This property in general does not hold for the (not diagonal) mass matrix.
  \deleted{However, neither of the preconditioners leads to bounded iteration counts.}

\begin{table}[ht!]
  \begin{center}
  \caption{Convergence of the preconditioned CG method for the singular elasticity
    problem. Positive definite preconditioners using respectively the mass and
    identity matrices to get a nonsingular system are considered. The maximum
    number of iterations is set to 150. The iterations are unbounded in both
    cases. Solutions due to preconditioner using the identity matrix converge at
    nearly optimal rate.}
\footnotesize{
  \begin{center}
\begin{tabular}{l|lc|lc|lc|lc}
\hline
  \multirow{3}{*}{size} & \multicolumn{4}{c|}{AMG($\la{A}+\la{M}$)} 
                        & \multicolumn{4}{c}{AMG($\la{A}+\la{I}$)}\\
  \cline{2-9}
  & \multicolumn{2}{c}{kernel not removed} & \multicolumn{2}{c|}{kernel removed} & 
\multicolumn{2}{c|}{kernel not removed} & \multicolumn{2}{c}{kernel removed} \\
  \cline{2-9}
& $\norm{u-u_h}_1$ & \# & $\norm{u-u_h}_1$ & \# & $\norm{u-u_h}_1$ & \# & $\norm{u-u_h}_1$ & \#\\
\hline
  2187    & 1.97E-02(0.26)  & 17  & 5.08E-03(0.99) & 17  & 1.11E-02(0.87)  & 21 & 5.08E-03(0.99)  & 21\\
  14739   & 2.58E-02(-0.39) & 19  & 2.29E-03(1.15) & 19  & 2.87E-03(1.95)  & 35 & 2.29E-03(1.15)  & 35\\
  107811  & 2.80E-02(-0.12) & 34  & 1.06E-03(1.11) & 34  & 1.21E-03(1.24)  & 81 & 1.06E-03(1.11)  & 81\\
  823875  & 2.82E-02(-0.01) & 53  & 5.12E-04(1.04) & 54  & 6.32E-04(0.94)  &
  $>$150 & 5.12E-04(1.04)  & $>$150\\
\hline
\end{tabular}
  \end{center}
}
\label{tab:krylov_AM}
\end{center}
\end{table}
\end{example}

Examples \ref{ex:pin_poisson}--\ref{ex:krylov_AM} have illustrated some of the
issues that might be encountered when solving the singular problem \eqref{eq:strong} 
with the finite element method. In particular, the following questions may be posed: 
(i) What is the cause of the poor convergence properties of pinpointing? (ii) What should be 
the order optimal preconditioner for CG? (iii) What should be the order optimal preconditioner 
for the Lagrange multiplier formulation? 

With questions (ii) and (iii) answered in detail in the remainder of the
text let us briefly comment on the first question. As will become apparent, the singular 
problem with a known kernel, such as \eqref{eq:strong}, 
possesses all the information necessary to formulate a well-posed problem and a 
convergent numerical method. In this sense, coming up with a datum to be prescribed 
in the pinpointed nodes is theoretically redundant, but usually required for implementation. 
Further, as pointed out in \cite{bochev} there are stability issues with prescribing 
point values of $H^1$ functions for $d\geq2$. However, we have not explored 
settings of HypreAMG \added{or other realizations of the preconditioner} that
could potentially improve convergence properties of the method in Example \ref{ex:pin_elasticity}.
\added{In this sense the two level preconditioner of \cite{vanek1999two}
  is interesting as the proposed method results in bounded CG iterations even
  with the variationaly problematic point boundary conditions.}

\section{Lagrange multiplier formulation}\label{sec:lagrange}
Let $Z\subset V=\left[H^1(\Omega)\right]^3$ denote the space of rigid motions of 
$\Omega$, For compatible data a unique solution $u$ of \eqref{eq:strong} 
is required to be linearly independent of functions in $Z$. To this end a Lagrange 
multiplier ${p}\in Z$ is introduced which enforces orthogonality of $u$ with respect to $Z$. 
The constrained variational formulation of
\eqref{eq:strong} seeks $u\in V, {p}\in Z$ such that
\footnote{Note that $(h, v)_{\partial\Omega}$ stands for the integral
$\int_{\partial\Omega} h \cdot v \,\mathrm{d}s$.}

\begin{equation}\label{eq:weak_long}
  \begin{aligned}
    &2\mu(\epsilon(u), \epsilon(v))+\lambda(\nabla\cdot u, \nabla\cdot v) - (p, v) = (f, v) +
    (h, v)_{\partial\Omega} \quad &\forall v\in V,\\ 
    -&(u, q)                         = 0 \quad &\forall{q}\in Z.
  \end{aligned}
\end{equation}
Equation \eqref{eq:weak_long} defines a saddle point problem for $(u, p)\in W$, $W=V\times Z$ 
satisfying
\begin{equation}\label{eq:weak}
  \mathcal{A}\begin{pmatrix}u\\p\end{pmatrix}=
\begin{pmatrix}
  A & B\\
  \dual{B} & 
\end{pmatrix}
\begin{pmatrix}u\\p\end{pmatrix}
=
\begin{pmatrix}l\\ 0 \end{pmatrix}
\end{equation}
where $l\in\dual{V}$ such that $\brack{l, v}=(f, v)+(h, v)_{\partial\Omega}$ 
and operators 
$A:V\rightarrow\dual{V}$, $B:Z\rightarrow\dual{V}$ are defined in terms of bilinear
forms
\begin{equation}\label{eq:form_def}
  a(u, v) = 2\mu(\epsilon(u), \epsilon(v))+\lambda(\nabla\cdot u, \nabla\cdot v)
  \quad\text{and}\quad
  b(u, q)=(u, q)
\end{equation}
as $\brack{Au, v}=a(u, v)$ and 
 $\brack{Bq, u}=-b(u, q)$. 
We note that in \eqref{eq:weak} operator $\dual{B}$ is the adjoint of $B$.

Existence and uniqueness of the solution to \eqref{eq:weak} follows from the Brezzi
theory \cite{brezzi}, see also \cite[ch 3.4]{braess}. The proof shall utilize the inequalities 
given in Lemma~\ref{lm:help}.
\begin{lemma} \label{lm:help}
  Let $u\in V$ arbitrary and $\omega(u)$ be the skew symmetric part of the displacement gradient $\nabla u$,
  \added[id=rev2, remark={comment 2}]{i.e. $\omega(u)=\tfrac{1}{2}((\nabla u)-\transp{(\nabla u)})$.} 
  Then
  \begin{subequations}
    \begin{align}
      &\norm{\epsilon(u)}\leq\norm{\nabla u} \quad\text{and}\quad\norm{\omega(u)}\leq\norm{\nabla u}, \label{eq:help1}\\
&\norm{\nabla\cdot u}\leq\sqrt{3}\norm{\nabla u},\label{eq:help2}\\
      &\exists C=C(\Omega)\mbox{ such that }\norm{z}_1\leq C\norm{z}\quad \forall
      z\in Z. \label{eq:help3}
  \end{align}
  \end{subequations}
  \begin{proof}
    Inequality \eqref{eq:help1} follows from the orthogonal decomposition 
    $\nabla u=\epsilon(u)+\omega(u)$. Inequality \eqref{eq:help2} follows by direct
    calculations. To establish the final inequality we first note that \eqref{eq:help3}
    clearly holds for rigid motions that are translations with constant $C=1$. To verify it for rigid
    rotations we consider the representation $z=Sx$ for some arbitrary skew-symmetric matrix 
    $S\in\reals^{3\times 3}$. Then 
    by definition $\omega(Sx)=S$ so that $(\omega(z),
    \omega(z))=\semi{S}^2\semi{\Omega}$ \added[id=rev2, remark={comment 3}]{with $\semi{S}=\sqrt{\tr{(\transp{S}S)}}$ the Frobenius
    norm.} In turn 
    \begin{equation}\label{eq:inertia_c}
 \norm{z}^2=(Sx, Sx)=\semi{S}^2(x, x)=\frac{(x, x)}{\semi{\Omega}}(\omega(z), \omega(z))=
      c(\Omega)\norm{\nabla z}^2, \quad c(\Omega)=\frac{(x, x)}{\semi{\Omega}}
  \end{equation}
  as $\epsilon(z)=0$. Therefore \eqref{eq:help3} holds for all rotations. We remark that 
    the constant $c$ in \eqref{eq:inertia_c} is related to the moment of inertia of the body. 
    Finally the statement follows with a constant $C(\Omega)=\sqrt{1+c(\Omega)}$ from 
    the decomposition of any $z\in Z$ into translations and rotations.
  \end{proof}
\end{lemma}
\begin{theorem}\label{thm:h1_norm} Let $f, h$ such that $l\in\dual{V}$. Then there 
exists a unique solution $u\in V$, ${p}\in Z$ of \eqref{eq:weak}. 
\begin{proof}
  We proceed by establishing the Brezzi constants. First, the bilinear form $a$
  is shown to be bounded with respect to the $\norm{\cdot}_1$. Indeed, by
  Cauchy-Schwarz inequality and inequalities \eqref{eq:help1}, \eqref{eq:help2} 
  we have for any $u, v\in V$
  \[
\begin{split}
a(u, v) = 2\mu(\epsilon(u), \epsilon(v)) + \lambda(\Div{u}, \Div{v})
  &\leq 2\mu\norm{\epsilon(u)}\norm{\epsilon(u)}+\lambda\norm{\Div{u}}\norm{\Div{v}}\\
  &\leq (2\mu+3\lambda)\norm{\nabla v}\norm{\nabla u}
  \leq \alpha^{*} \norm{u}_1\norm{v}_1
\end{split}
  \]
with $\alpha^{*}=2\lambda+3\mu$.
  Ellipticity of $a$ on 
  $\Zperp=\set{v\in V; (v, z)=0\,\forall z\in Z}=\set{v\in V; b(v,
  {p})=0\,\forall {p}\in Z}$
  follows from Korn's inequality \eqref{eq:kornZ}. Since $\lambda\geq 0$ by assumption
  \[
    a(u, u) = 2\mu\norm{\epsilon(u)}^2 + \lambda\norm{\Div u}^2\geq
  2\mu\norm{\epsilon(u)}^2 \geq \alpha_*\norm{u}^2_1\quad \forall u\in\Zperp,
  \]
with $\alpha_*=2\mu C$ and $C=C(\Omega)$ the constant from \eqref{eq:kornZ}.
Boundedness of $b$ with a constant $\beta^*=1$ follows from the Cauchy-Schwarz inequality. 
Finally, using \eqref{eq:help3} we have for arbitrary $p\in Z$
  \[
\sup_{v\in V} \frac{b(v, {p})}{\norm{v}_1} 
  \geq
  \frac{(p, p)}{\norm{p}_1}
  \geq
  \frac{\norm{p}^2}{C\norm{p}}
  =\frac{1}{C}\norm{p}
  \]
  so that the inf-sup condition holds with $\beta_*=C^{-1}$ with $C$ the constant from
  \eqref{eq:help3}.
\end{proof}
\end{theorem}

We remark that Theorem \ref{thm:h1_norm} implies that the operator
$\mathcal{A}:W\rightarrow \dual{W}$ from \eqref{eq:weak} is an isomorphism. In 
particular, conditions \eqref{eq:compat} need not to hold for there to exist a 
unique solution of \eqref{eq:weak_long}.

In order to find the solution of the well-posed \eqref{eq:weak} numerically,
conditions from Theorem \ref{thm:h1_norm} must hold with discrete subspaces
$V_h$, $Z_h$, see \cite{fortin} or \cite[ch 3.4]{braess}. Typically, satisfying the 
discrete inf-sup condition presents an issue and requires 
choice of compatible finite element discretization of the involved spaces, e.g. 
Taylor-Hood or MINI elements \cite{arnold} for the Stokes equations. For the 
conforming discretization $V_h\subset V$, $Z_h=Z$ the following result shows that 
the discrete inf-sup condition holds.
\begin{theorem}\label{thm:dlbb} Let $Z_h=Z$, $V_h\subset V$ and $b$ the 
  bilinear form defined in \eqref{eq:form_def}. Then there is a constant 
  $\beta_*$ independent of $h$ such that
  $\inf_{p\in Z_h} \sup_{v\in V_h} \tfrac{b(v, {p})}{\norm{v}_1 \norm{p}} \geq \beta_{*}$.
\end{theorem}
  
\begin{proof} The proof mirrors the continuous inf-sup condition in Theorem \ref{thm:h1_norm}. 
  Let $p\in Z_h$ be given. Since $Z=Z_h\subset V_h$ we get by taking $v=p$
  \[
\sup_{v\in V_h} \frac{b(v, {p})}{\norm{v}_1} 
  \geq
  \frac{(p, p)}{\norm{p}_1}
  \geq
  \frac{\norm{p}^2}{C\norm{p}}
  =\frac{1}{C}\norm{p},
  \]
where $C$ is the constant from \eqref{eq:help3}.
\end{proof}

Following Theorems \ref{thm:h1_norm}, \ref{thm:dlbb} and operator preconditioning 
\cite{kent, malek} the Riesz map $\mathcal{B}_1:\dual{W}\rightarrow W$ with 
respect to inner product $(u, v)_1+(p, q)$ with $(u, p), (v, q)\in W$ 
\begin{equation}\label{eq:B1}
  \mathcal{B}_1 = \inv{\begin{pmatrix}
                  H & \\
                    & I
                  \end{pmatrix}},
  \quad
  H:V\rightarrow\dual{V}, \brack{Hu, v}=(u, v)_1
  \quad\text{and}\quad
  I:Z\rightarrow \dual{Z}, \brack{Ip, q}=(p, q)
\end{equation}
defines a preconditioner for discretized \eqref{eq:weak} whose condition number 
is independent of $h$. This follows from Brezzi constants in Theorems 
\ref{thm:h1_norm}, \ref{thm:dlbb} being free of the discretization parameter. 

  Since applying the preconditioner \eqref{eq:B1} requires an inverse of the $6\times 6$ 
  mass matrix of the space rigid motions it is advantageous to chose a basis of $Z$ 
  in which the matrix is well-conditioned. With the choice of an $L^2$ orthonormal 
  basis the obtained mass matrix is an identity and we shall therefore briefly discuss
  construction of such a basis.

\subsection{Construction for orthonormal basis of rigid motions}\label{sec:make_rm}
Consider a unit cube $\Omega=\left[-\tfrac{1}{2}, \tfrac{1}{2}\right]^3$
centered at the origin. Denoting ${e}_i$, $i=1, 2, 3$ the canonical unit vectors 
the set
\[
  Z_{\mbox{\mancube}} = \set{{e}_1, {e}_2, {e}_3,
                              x\times{e}_1, x\times{e}_2, x\times{e}_3}
\]
constitutes an orthonormal basis of the rigid motions of $\Omega$ with respect to 
the $L^2$ inner product. Clearly, the basis for an arbitrary body can be obtained from 
$Z_{\mbox{\mancube}}$ by a Gram-Schmidt process. However, we shall advocate here a 
construction derived from physical considerations. The construction was
originally presented by the authors in \cite{mekit_rm}.

\begin{lemma}\label{lm:basis} Let ${c}=\inv{\semi{\Omega}}(x, 1)$ be the center of mass of
  $\Omega$, $I_{\Omega}$ the tensor of inertia \cite[ch 4.]{gurtin} of $\Omega$ 
  with respect to $c$ 
  \[
I_{\Omega}=\int_{\Omega} I\transp{(x-{c})}(x-{c})+(x-{c})\otimes(x-{c})\,\mathrm{d}x
  \]
and $(\lambda_i, {v}_i)$, $i=1, 2, 3$ the eigenpairs of the tensor. Then the 
  set
  \begin{equation}\label{eq:basis}
    Z_{\Omega} = \set{
\semi{\Omega}^{\nhalf}v_1, \semi{\Omega}^{\nhalf}v_2, \semi{\Omega}^{\nhalf}v_3,
\lambda_1^{\nhalf}(x-{c})\times v_1, \lambda_2^{\nhalf}(x-{c})\times v_2,
\lambda_3^{\nhalf}(x-{c})\times v_3}
  \end{equation}
is the $L^2$ orthonormal basis of rigid motions of $\Omega$.
\begin{proof}
Note that by construction $I_{\Omega}$ is a symmetric positive definite 
tensor. Thus $\lambda_i>0$ and there exists a complete set of eigenvectors 
$\transp{v_i}v_j=\delta_{ij}$. We proceed to show that the Gram matrix of the 
proposed basis is an identity. First $(v_i, v_j)=\semi{\Omega}\delta_{ij}$ by 
orthonormality of the eigenvectors. Further, for $((x-c)\times v_i, v_j)=(v_i\times v_j, (x-c))$ 
and in the nontrivial case $i\neq j$ the product is zero since $c$ is the center of mass. 
Finally $((x-{c})\times v_i, (x-{c})\times v_j)=\transp{v_i}I_{\Omega}v_j=\lambda_i\delta_{ij}$.
\end{proof}
\end{lemma}
We remark that the rigid motions of the body are in the constructed basis given
in terms of translations along and rotations around the principal axes of the 
tensor that describes its rotational kinetic energy. 

Note also, that the construction can be generalized to yield an orthonormal basis with 
respect to different inner products. In particular, let $Z_h=\spn{\set{z_k}_{k=1}^6}\subset V_h$ 
be functions approximating some basis of $Z$.
For $u, v\in V_h$ let 
$\la{u}=\pi_h u$,  $\la{v}=\pi_h v$ be coefficient vectors in the nodal basis of $V_h$. The $l^2$ orthonormal basis of $Z_h$ 
can be created using Lemma \ref{lm:basis} by replacing $(u, v)$ with
$\transp{\la{u}}\la{v}$. The differences between the bases are shown in Figure
\ref{fig:domains} where the defining principal axes of the $L^2$ and $l^2$ orthonormal 
basis of rigid motions are drawn. If $\Omega$ is uniformly triangulated the
bases are practically identical. However, the $l^2$ basis 
changes in the presence of a non-uniform mesh refinement. 

\begin{figure}
  \begin{center}
  \includegraphics[width=0.45\textwidth]{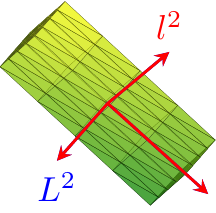}
  \includegraphics[width=0.45\textwidth]{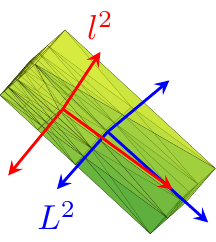}
  \end{center}
\caption{Computational domains considered in the numerical examples for linear
elasticity are obtained by uniformly refining the parent mesh. (Left) Parent is 
  close to uniformly triangulated. (Right) The parent mesh is refined near a single
edge of the domain. The blue and red arrows indicate the principal axes of the tensor
$I_\Omega$, cf. Lemma \ref{lm:basis}, defined using the $L^2$ and $l^2$ inner products. 
Axes are drawn from the center of mass computed using the respected inner
products. Only the $L^2$ basis is stable upon change of triangulation from
  uniform (left) to nonuniform (right).
}
  \label{fig:domains}
\end{figure}

Formulation of the problem \eqref{eq:weak_long} with respect to 
an $L^2$ orthonormal basis $\set{z_k}_{k=1}^{6}$ of the space of rigid motions results 
in the mapping between $Z$ and $\reals^6$ being an isometry. In turn, if the
discretized problem is considered with space $V_h\times \reals^6$ and its natural norm, 
the Brezzi constants will be those obtained in Theorem \ref{thm:dlbb}. On the other hand, 
for a non-orthonormal basis only equivalence between the norms holds: There exists 
$C_1, C_2>0$ such that for all $p\in Z$
\[
  C_1 \semi{c} \leq \norm{p} \leq C_2 \semi{c}, \quad p=\sum_{k=1}^6 c_k z_k
\]
and the constants $C_1, C_2$ enter the estimates in the Brezzi theory. For an
unfortunate choice of the basis it is then possible that $C_1=C_1(h)$ or $C_2=C_2(h)$
leading to mesh dependent performance of a preconditioner using the $l^2$ norm for
(the Lagrange multiplier space) $\reals^6$.


Returning to preconditioner \eqref{eq:B1} recall that the Brezzi constants 
$\alpha^{*}, \alpha_{*}$ depend on the Lam{\' e} constants and thus
$\mathcal{B}_1$ does not define a parameter robust preconditioner. To address the 
dependence on material parameters, we shall at first assume that $\mu$ and $\lambda$ 
are comparable in magnitude. The case $\lambda \gg \mu$ is postponed until \S \ref{sec:mixed}.


\subsection{Robust preconditioning of the singular problem} 

Parameter robust preconditioners for the Lagrange multiplier formulation 
of the singular elasticity problem \eqref{eq:weak_long} can be analyzed by the
operator preconditioning framework of \cite{kent}. The preconditioners are constructed by 
considering \eqref{eq:weak} in parameter dependent spaces, e.g. \cite{bergh}, which are
equivalent with $V$ as a set, but the topology of the spaces is given by different, 
parameter dependent, norms. Two such norms leading to two different
preconditioners are constructed next.

For $u\in V$ consider the orthogonal decomposition $u=u_Z+u_{Z^\perp}$ where 
$u_z\in Z$ and $u_{z^\perp}\in Z^{\perp}$. Bilinear forms $(\cdot, \cdot)_E$, 
$(\cdot, \cdot)_M$ over $V$ are defined in terms of $A$ from \eqref{eq:weak} and 
operators $Y:V\rightarrow\dual{V}$, $M:V\rightarrow\dual{V}$ as
\begin{equation}\label{eq:Vinner}
  \begin{aligned}
    &\brack{Yu, v} = (u_Z, v_Z), &\quad\quad(u, v)_E = \brack{Au, v} + \brack{Yu, v}, \\
    &\brack{Mu, v} = (u, v),           &\quad\quad(u, v)_M = \brack{Au, v} + \brack{Mu, v}.
  \end{aligned}
\end{equation}
The forms \eqref{eq:Vinner} define functionals $\norm{\cdot}_E$ and 
$\norm{\cdot}_M$ over $V$ such that
\begin{equation}\label{eq:Vnorm}
\norm{u}_E = \sqrt{(u, u)_E}
\quad\text{and}\quad
\norm{u}_M = \sqrt{(u, u)_M}.
\end{equation}
\begin{lemma}\label{lm:inners}
  Let $\norm{\cdot}_E$ and $\norm{\cdot}_M$ be the functionals \eqref{eq:Vnorm}.
  Then $\norm{\cdot}_E$ and $\norm{\cdot}_M$ define norms on $V$ which are 
  equivalent with the $H^1$ norm.
\end{lemma}
\begin{proof}
  From the orthogonal decomposition of $u\in V$ it follows that $\norm{u}^2_M=\norm{u}^2_E+\norm{u_{\Zperp}}^2$. 
  Together with Lemma \ref{lm:help} we thus establish
  \[
  \norm{u}^2_E \leq \norm{u}^2_M \leq (2\mu+3\lambda + 1)\norm{u}^2_1\quad\forall
  u\in V.
  \]
To complete the equivalence, let $C=C(\Omega)$ be the constant from Korn's inequality 
  \eqref{eq:kornV}. Then for all $u\in V$
  \[
  \norm{u}^2_M \geq 2\mu\norm{\epsilon(u)}^2 + \norm{u}^2 \geq c\norm{u}^2_1,
  \]
with $c=C$ for $2\mu>1$ and $c=2\mu C$ otherwise. Finally, for equivalence of the
  $E$-norm, the Korn's inequality on $\Zperp$, see \eqref{eq:kornZ} also 
  Theorem \ref{thm:h1_norm}, yields
  \[  
    \norm{u}^2_E = 2\mu\norm{\epsilon(u)}^2+\lambda\norm{\Div u}^2 \geq 
                   2\mu C\norm{u}^2_1\quad \forall u\in\Zperp
  \]
  with $C=C(\Omega)$, while using \eqref{eq:help3} in Lemma \ref{lm:help} gives
  \[
    \norm{u}_E=\norm{u}\geq C_1(\Omega) \norm{u}_1
  \]
for any $u\in Z$. Thus $E$ and $H^1$ norms are equivalent on $\Zperp$ and $Z$
respectively. The proof is completed by observing \added[id=rev2, remark={comment 4}]{that $u_Z$ and $u_{\Zperp}$
satisfy $(u_Z, u_{\Zperp})_E=0$ so that}
  \[
    \begin{split}
      \norm{u}^2_E = 2\mu\norm{\epsilon(u_{\Zperp})}^2+\lambda\norm{\Div
      u_{\Zperp}}^2 + 
      \norm{u_Z}^2 &\geq 2\mu C \norm{u_{\Zperp}}^2_1 + C_1 \norm{u_Z}^2_1\\
                   &\geq c(\norm{u_{\Zperp}}^2_1+\norm{u_Z}^2_1),
    \end{split}
  \]
$c=\min(2\mu C, C_1)$, while for the $H^1$ inner product 
  $\norm{u}^2_1\leq 2(\norm{u_{\Zperp}}^2_1+\norm{u_Z}^2_1)$ 
  holds. Thus $\norm{u}_E^2\geq \tfrac{c}{2}\norm{u}_1^2$ for all $u\in V$.
\end{proof}

Using equivalent norms of $V$ from Lemma \ref{lm:inners} we readily establish
equivalent norms for the product space $W=V\times Z$
\begin{equation}\label{eq:EM}
  \norm{w}_E=\norm{(u, p)}_E = \sqrt{\norm{u}_E^2+\norm{p}^2}
\quad\text{and}\quad
  \norm{w}_M=\norm{(u, p)}_M = \sqrt{\norm{u}_M^2+\norm{p}^2}
\end{equation}
and consider as preconditioners for \eqref{eq:weak} the operators
$\mathcal{B}_E:\dual{W}\rightarrow W$ and $\mathcal{B}_M:\dual{W}\rightarrow W$
\begin{equation}\label{eq:BEM}
  \mathcal{B}_E = 
  \inv{\begin{pmatrix}
  A + Y &\\
        & I
  \end{pmatrix}}
\quad\text{and}\quad
  \mathcal{B}_M = 
  \inv{\begin{pmatrix}
  A + M &\\
        & I
  \end{pmatrix}}.
\end{equation}
Note that the mappings \eqref{eq:BEM} are the Riesz maps with
respect to the inner products which induce norms \eqref{eq:EM}. We proceed with
analysis of the properties of $\mathcal{B}_E$.

\begin{theorem}\label{thm:e_norm}
Let $\mathcal{A}:W\rightarrow\dual{W}$ be the operator and the space from \eqref{eq:weak} 
  and $W_E$ be the space $W$ considered with $\norm{\cdot}_E$ norm \eqref{eq:EM}. 
  Then $\mathcal{A}:W_E\rightarrow\dual{W_E}$ is an isomorphism. Moreover the Riesz map 
  $\mathcal{B}_E:\dual{W_E}\rightarrow {W_E}$ in \eqref{eq:BEM} defines the canonical 
  preconditioner for \eqref{eq:weak}.
\end{theorem}
  \begin{proof}
We shall show that the first assertion holds by establishing the Brezzi constants. 
    Recall the definition of the bilinear form $a$ given in \eqref{eq:form_def}. Then,
    by the Cauchy-Schwarz inequality and \eqref{eq:help1} in Lemma \ref{lm:help}, the
    inequality $a(u, v)\leq \sqrt{a(u, u)}\sqrt{a(v, v)}$ holds for any $u, v\in V$. 
    In turn for all $u, v\in V$
\[
  a(u, v) \leq \sqrt{a(u, u)}\sqrt{a(v, v)}
          \leq \sqrt{a(u, u)+(u_Z, u_Z)}
               \sqrt{a(v, v)+(v_Z, v_Z)} = \norm{u}_E\norm{v}_E
\]
and $a$ is bounded with respect to $E$ norm with a constant $\alpha^*=1$.
Further, $u_Z=0$ for $u\in\Zperp$. Hence $a(u, u)=a(u, u)+(u_Z, u_Z)=\norm{u}^2_E$
for all $u\in\Zperp$ and the form is $E$ elliptic on $\Zperp$ with constant $\alpha^*=1$. 
To compute the boundedness constant of the form $b$, the orthogonal decomposition
$u=u_Z+u_{\Zperp}$ is used so that for all $u\in V$, $q\in Z$
    \[
      b(u, q) = (u_Z+u_{\Zperp}, q) = (u_Z, q) \leq \norm{u_Z}\norm{q}
    =\sqrt{a(u, u)+\norm{u_Z}^2}\norm{q} = \norm{u}_E\norm{q}
    \]
  and we have $\beta^{*}=1$. Finally, taking any $q\in Z$ and setting $u=q$ in 
\[
\sup_{u\in V} \frac{b(u, {q})}{\norm{u}_E}
\geq
\frac{(q, q)}{\sqrt{a(q, q)+(q_Z, q_Z)}}
=
    \frac{\norm{q}^2}{\sqrt{0+\norm{q}^2}}
\geq
    \norm{q}
\]
and thus the inf-sup condition holds with $\beta_{*}=1$.
As all the constants are independent of material parameters, the second assertion 
follows from the first one by operator preconditioning \cite[ch 5.]{kent}.
  \end{proof}
Using Theorem \ref{thm:e_norm} it is readily established that the condition
number of the composed operator $\mathcal{B}_E \mathcal{A}:W\mapsto W$ is equal 
to one. We further note that discretizing operator $\mathcal{B}_E$ leads to discrete 
nullspace  preconditioners of \cite[ch 6.]{liesen}. 

While the spectral properties of $\mathcal{B}_E$ are appealing, the preconditioner is
impractical. Consider $\la{B}_E$ as a matrix representation of the Galerkin 
approximation of $\mathcal{B}_E$ in $W_h\subset W$. Then 
$\la{B}_E=\inv{\text{diag}(\la{A}+\la{Y}\transp{\la{Y}}, \la{I})}$ where
$\la{Y}=\reals^{n\times 6}$, $\la{y}_k=\text{col}_k\la{Y}=\pi_h z_k$ and 
$z_k\in V_h$ is the function from the $L^2$ orthogonal basis of the space of rigid motions. 
Due to the second (nonlocal) term the matrix $\la{A}+\la{Y}\transp{\la{Y}}$ is dense. Further, as
shall be discussed in \S\ref{sec:krylov}, inverting the operator requires computing 
(the action of) the pseudoinverse of the singular matrix $\la{A}$. The mapping
$\mathcal{B}_M$, on the other hand, leads to a more practical preconditioner.
\begin{theorem}\label{thm:m_norm} Let $\mathcal{A}:W\rightarrow\dual{W}$ be the
  operator and space defined in \eqref{eq:weak} and $W_M$ be defined analogically
  to Theorem \ref{thm:e_norm}. Then $\mathcal{A}:W_M\rightarrow\dual{W_M}$ 
  is an isomorphism. Moreover the Riesz map $\mathcal{B}_M:\dual{W_M}\rightarrow W_M$ in
\eqref{eq:BEM} defines a parameter robust preconditioner for \eqref{eq:weak}.
\end{theorem}
\begin{proof}
As in the proof of Theorem \ref{thm:e_norm} we establish that $a(u, v)\leq\norm{u}_M\norm{v}_E$ 
for all $u, v\in V$ and $b(v, p)\leq \norm{v}\norm{p}\leq \norm{v}_M\norm{p}$ for 
all $v\in V$, $p\in Z$. Setting $v=p\in Z$ then yields
  $\inf_{p\in Z}\sup_{v\in V}\tfrac{b(v, p)}{\norm{v}_M\norm{p}}\geq 1$.
  For $M$ ellipticity of $a$ on $\Zperp$, assume existence of $C=C(\Omega)$ such
  that $\norm{u}^2\leq C\norm{\epsilon(u)}^2$ for all $u\in \Zperp$. Then on $\Zperp$
  \[
    \norm{u}^2\leq C\norm{\epsilon(u)}^2\leq C\mu\norm{\epsilon(u)}^2 \leq
    C(2\mu\norm{\epsilon(u)}^2+\lambda\norm{\Div u}^2)=C\norm{u}^2_E
  \]
and
  \[
    \norm{u}^2_M = \norm{u}^2_E + \norm{u}^2 \leq (C+1)\norm{u}^2_E
  \]
so that $a(u, u)=\norm{u}^2_E\geq\inv{(1+C)}\norm{u}^2_M$. Finally we comment on
the assumption of existence of the constant $C$. Assume the contrary. Then there
is $u\in\Zperp$ such that $\norm{e(u)}=1$, $\norm{w(u)}=0$ and the $\norm{u}$
  unbounded. However, such $u$ violates Korn's inequality \eqref{eq:kornZ}.
\end{proof}

We remark that Theorem \ref{thm:m_norm} required an additional assumption 
$2\mu\geq 1$. The assumption is not restrictive as it can be always achieved by
scaling the equations such that the inequality is satisfied. Note also that with
the orthonormal basis of rigid motions the discrete 
  preconditioner based on $\mathcal{B}_M$ is such that
  $\inv{\la{B}_M}=\text{diag}(\la{A}+\la{M}, \la{I})$, with
$\la{M}$ the mass matrix. The system to be assembled is therefore sparse. 

Following Theorem \ref{thm:m_norm} the condition number of the preconditioned 
operator $\mathcal{B}_M \mathcal{A}:W\rightarrow W$ depends solely on the constant $C$ from Korn's 
  inequality \eqref{eq:kornZ}. An approximation for the constant is provided by 
  the smallest positive eigenvalue $\lambda^{+}_{\text{min}}$ of the problem
\[
  \begin{pmatrix}
  \la{A}           & \la{B}\\
  \transp{\la{B}}  &
  \end{pmatrix}\begin{pmatrix}\la{u}\\
                            \la{p}\end{pmatrix}=
                            \lambda
  \begin{pmatrix}
    \la{A} + \la{M} &   \\
                    & \la{I}
  \end{pmatrix}\begin{pmatrix}\la{u}\\
                            \la{p}\end{pmatrix}.
\]
In Table \ref{tab:eigenvalues}, Appendix \ref{sec:appendix}, the constant has
been computed for two different domains; a cube from Example
\ref{ex:pin_elasticity} and a hollow cylinder. In both cases $C \approx 1$ can
be observed.

In order to demonstrate $h$ robust properties of $\mathcal{B}_M$, the problem from
Example \ref{ex:pin_elasticity} is considered with basis from \S \ref{sec:make_rm} 
and discretized on $V_h\subset V$. The resulting preconditioned linear system is 
solved by the minimal residual (MinRes) method \cite{minres} as implemented
in \textit{cbc.block}, the FEniCS library for block matrices \cite{block}
using as the preconditioner
\[
  \la{B}_M=
  \begin{pmatrix}
    \text{AMG}(\la{A} + \la{M}) &   \\
                                & \la{I}
  \end{pmatrix}.
\]
\added[id=rev1, remark={comment 4}]{More specifically, the preconditioner
  uses a single AMG $V$ cycle with one pre and post smoothing by a symmetric SOR
  smoother. The rigid motions were not passed to the routine on initialization.
}
The saddle point system was assembled and inverted\footnote{Implementation of the
solver as well as the two algorithms discussed in \S\ref{sec:energy} and
\S\ref{sec:mixed} can be found online at \url{https://github.com/MiroK/fenics-rigid-motions}.
}using \textit{cbc.block}. The results of the experiment are 
presented in Table \ref{tab:krylov_lagrange_cmp}. Clearly, the number of iterations 
required for convergence is independent of the discretization. Moreover, the method 
yields numerical solutions which converge in the $H^1$ norm at the optimal
rate\footnote{We recall that $V_h$ is constructed from continuous linear
Lagrange elements.} on both the uniform and nonuniform meshes, cf. Figure \ref{fig:domains}. 

A drawback of the Lagrange multiplier formulation is the cost of solving the 
resulting indefinite linear system.
Following e.g.
\cite[ch 7.2]{greenbaum1997iterative} let the condition number of a Hermitian matrix $\la{A}$ be
$\kappa(\la{A})=\sfrac{\lambda_{\text{max}}(\la{A})}{\lambda_{\text{min}}(\la{A})}$ where
$\lambda_{\text{max}}(\la{A})$, $\lambda_{\text{min}}(\la{A})$ are respectively the largest
and smallest (in magnitude) eigenvalues of the matrix. For $\la{A}$ Hermitian indefinite and 
under simplifying assumptions on the spectrum \cite[ch 3.2]{liesen_tichy} gives the following 
bound on the relative error in residual $r_n$ at step $n$ of the MinRes method
\[
  \frac{\semi{r_n}}{\semi{r_0}}
  \leq 2\left(\frac{\kappa(\la{A})-1}{\kappa(\la{A})+1}\right)^{\floor{n/2}}.
\]
The result should be contrasted with a similar one for the error $e_n$ at the 
$n$-th step of CG method on symmetric positive definite matrix $\la{A}$, e.g. 
\cite[thm 38.5]{trefethen},
\[
  \frac{\transp{e_n}\la{A}e_n}{\transp{e_0}\la{A}e_0}
  \leq 2\left(\frac{\sqrt{\kappa(\la{A})}-1}{\sqrt{\kappa(\la{A})}+1}\right)^{n}.
\]
%


While the above estimates are known to give the worst case behavior of the two 
methods, the faster rate of convergence of CG motivates investigating formulations 
of \eqref{eq:strong} to which the conjugate gradient method can be applied. 

\section{Conjugate gradient method for discrete singular problems}\label{sec:krylov}
We consider a variational formulation of \eqref{eq:strong}: Find $u\in
V=\left[H^1(\Omega)\right]^3$ such
that
\begin{equation}\label{eq:cg_long}
2\mu(\epsilon(u), \epsilon(v))+\lambda(\nabla\cdot u, \nabla\cdot v) = 
(f, v)+(h, v)_{\partial\Omega}\quad \forall v\in V.
\end{equation}
Denoting $a:V\times V\rightarrow\reals, l:\dual{V}\rightarrow\reals$ the bilinear
and linear forms defined by \eqref{eq:cg_long}, we note that the problem is not well-posed in
$V$. Indeed, the compatibility conditions \eqref{eq:compat} restrict the
functionals for which the solution can be found to
$l\in\Zpolar=\set{f\in\dual{V};\brack{f, z}=0\,\forall z\in Z}$. 
Moreover, only the part of $u$ in $\Zperp$ is uniquely determined by
\eqref{eq:cg_long}. More precisely we have the following result. 

\begin{theorem}\label{thm:cg}
Let $l\in\Zpolar$. Then there exists a unique solution of the problem
  \begin{equation}\label{eq:cg_zperp}
\text{Find }u\in\Zperp\text{ such that for any }v\in\Zperp\text{ it holds that
    }a(u, v) = \brack{l, v}.
  \end{equation}
\end{theorem}
\begin{proof}
The complete proof can be found as Theorem 11.2.30 in \cite{brenner}. Note that
boundedness and ellipticity of $a$ on $\Zperp$ with $\norm{\cdot}_1$ are proven 
as part of Theorem \ref{thm:h1_norm}.
\end{proof}
We remark that if \eqref{eq:compat} holds then $u\in\Zperp$ solves \eqref{eq:cg_zperp} 
if and only if $(u, 0)$ solves the Lagrange multiplier problem \eqref{eq:weak}.
Further, the well-posed variational problem \eqref{eq:cg_zperp} is not suitable for
discretization by the finite element method as the approximation leads to a
dense linear system. A sparse discrete problem to which the conjugate gradient 
method shall be applied is therefore derived from \eqref{eq:cg_long}. 

Recall $\dim{Z}=6$, $n=\dim{V_h}$ and let $V_h=\spn{\set{\phi_i}_{i=1}^n}$. Discretizing
the variational problem \eqref{eq:cg_long} leads to a linear system
\begin{equation}\label{eq:Axb}
  \la{A}\la{u}=\la{b},
\end{equation}
where $\la{A}\in\reals^{n\times n}$ such that $A_{ij}=a(\phi_j, \phi_i)$ and
vector $\la{b}\in\reals^n$, $b_i=\brack{l, \phi_i}$. Note that we shall consider
\eqref{eq:Axb} for a general right hand side, that is, not necessarily a
discretization of $l\in\Zpolar$. We proceed by reviewing properties of the discrete 
system.

Due to symmetry and ellipticity of the bilinear form $a$ on $\Zperp$ there exists 
respectively $6$ vectors $\la{z}_k$ and $n-6$ eigenpairs $(\gamma_i, \la{u}_i)$, 
$\gamma_i>0$ such that $\la{A}\la{z}_k=0$, $\transp{\la{z}_k}\la{u}_i=0$, 
$\la{A}\la{u}_i=\gamma_i\la{u}_i$ and $\transp{\la{u}_i}\la{u}_j=\delta_{ij}$. 
From the decomposition of $\la{A}$ it follows that the system \eqref{eq:Axb} is 
solvable if and only if $\transp{\la{z}_k}\la{b}=0$ for any $k$ and the unique 
solution of the system is $\la{u}\in\spn{\set{\la{u}_i}_{i=1}^{n-6}}$. We note that the 
last statement is the Fredholm alternative for \eqref{eq:Axb}. As a further
consequence of the decomposition it is readily verified that given compatible 
vector $\la{b}$, the solution of \eqref{eq:Axb} is $\la{u}=\la{B}_A\la{b}$ with
$\la{B}_A$ such that $\la{B}_A\la{y}=\sum_i\inv{\gamma_i}\left(\transp{\la{u}_i}\la{y}\right)\la{u}_i$.
The matrix $\la{B}_A$ is the pseudoinverse \cite{penrose} or natural inverse 
\cite[ch 3.]{lanczos} of $\la{A}$. 

We note that any vector from $\reals^n$ can be orthogonalized with
respect to the kernel of $\la{A}$ by a projector $\la{P}_Z=\la{I}-\la{Z}\transp{\la{Z}}$,
where $\la{Z}\in\reals^{n\times 6}$ is the matrix consisting of $l^2$ orthonormal 
basis vectors of the kernel.

With $\la{b}$ such that $\transp{\la{Z}}\la{b}=0$ the solution $\la{u}$ of linear 
system \eqref{eq:Axb} can be computed by the conjugate gradient method, e.g. 
\cite{shewchuk}. Let $\la{u}^0$ be the starting vector for the iterations. Then,
assuming exact arithmetic and no preconditioner, the method preserves the
component of $\la{u}^0$ in $\la{Z}$, i.e. $\transp{\la{Z}}\la{u}^0=\transp{\la{Z}}\la{u}$.
In particular, $\transp{\la{Z}}\la{u}^0=0$ is required to obtain a solution
orthogonal to the kernel. On the other hand, let $\la{B}$ be the CG
preconditioner. Then the iterations introduce components of the kernel to the
solution even if $\transp{\la{Z}}\la{u}^0=0$, unless the range of $\la{B}$ is 
orthogonal to $\la{Z}$. 

\subsection{Preconditioned CG for singular elasticity problem}
A suitable preconditioner for \eqref{eq:Axb} is obtained by a composition with the 
$\la{P}_Z$ projector and we shall consider $\la{B}_M=\la{P}_Z\inv{(\la{A}+\la{M})}$ 
where $\la{M}$ is the mass matrix. That the preconditioner leads to bounded iteration
count (and converging numerical solutions) is demonstrated in Table 
\ref{tab:krylov_cmp}, cf. left pane. The preconditioner is also compared with
a different preconditioner based on the approximation of the pseudoinverse $\la{B}_A$. 
The approximation can be constructed by passing 
a kernel of the operator to \deleted{the multigrid preconditioner} the CG routine,
in the form of the $l^2$ orthonormal basis vectors, see \emph{MatSetNullSpace} in
PETSc\cite{petsc}. Note that the preconditioners perform similarly in terms of iteration 
count, however, for large systems the pseudoinverse appears to be a faster method.

We remark that in terms of operator preconditioning, the preconditioner based on the
pseudoinverse can be interpreted as a Riesz map $\Zpolar\rightarrow\Zperp$
defined with respect to the inner product induced by the bilinear form $a$.
Recall that $a$ is symmetric and elliptic on $Z^{\perp}$. On the other hand
$\la{B}_M$ approximates a mapping $\Zpolar\rightarrow V\rightarrow\Zperp$.
\begin{table}
  \begin{center}
  \caption{
Preconditioned CG iterations on \eqref{eq:Axb} obtained by discretization of 
    \eqref{eq:cg_long} with problem parameters as in Example 
\ref{ex:pin_elasticity} and two preconditioners. Both systems are solved with 
relative tolerance of $10^{-10}$. Uniform mesh is used.
  }
\footnotesize{
\begin{tabular}{l|lcl|lcl}
\hline
  \multirow{2}{*}{size} & \multicolumn{3}{c|}{$\la{P}_z\text{AMG}(\la{A}+\la{M})$}
                        &  \multicolumn{3}{c}{$\text{AMG}(\la{A}|\la{Z})$}\\
  \cline{2-7}
 & $\norm{u-u_h}_1$ & \# & time $\left[s\right]$ &
   $\norm{u-u_h}_1$ & \# & time $\left[s\right]$\\
\hline
14739   & 1.14E-02 (1.09) & 22 & 0.491 & 1.14E-02 (1.09) & 21 & 0.537\\
107811  & 5.49E-03 (1.06) & 23 & 10.17 & 5.49E-03 (1.06) & 23 & 10.96\\
823875  & 2.71E-03 (1.02) & 24 & 103.5 & 2.71E-03 (1.02) & 25 & 86.51\\
6440067 & 1.35E-03 (1.00) & 26 & 1580  & 1.35E-03 (1.00) & 26 & 911.9\\
\hline
\end{tabular}
}
\label{tab:krylov_cmp}
\end{center}
\end{table}

Having established preconditioners for the indefinite system stemming from the 
Lagrange multiplier formulation \eqref{eq:weak} and the positive semi-definite
problem stemming from \eqref{eq:cg_long}, we shall finally discuss approximation
properties of the computed solutions. To this end the problem from Example
\ref{ex:pin_elasticity} is considered where $f$ is perturbed by rigid motions. 
Note that while with the new functional $l$ the problem \eqref{eq:weak} is
well-posed, in \eqref{eq:Axb} a compatible right hand side $\la{b}$ will be obtained
by projector $\la{P}_Z$.

Results of the experiment are listed in Table \ref{tab:krylov_lagrange_cmp}.
The Lagrange multiplier method converges at optimal rate on both uniformly 
and non-uniformly discretized mesh, cf. Figure \ref{fig:domains}. On the other hand, 
solutions to \eqref{eq:Axb} converge to the true solution \textit{only} on the 
uniform mesh while there is no convergence with nonuniform discretization. Note
that this is not signaled by growth of the iterations - for both
methods the iteration counts are bounded. Note also that MinRes takes about
twice as many iterations as CG.

\begin{table}[]
  \begin{center}
  \caption{
    (top) Convergence properties of the Lagrange multiplier formulation \eqref{eq:weak} 
    and (bottom) the singular formulation \eqref{eq:cg_long} utilizing $l^2$ orthogonal basis
of the nullspace to invert the system \eqref{eq:Axb}. Only the
multiplier formulation yields solutions converging on uniform and nonuniform meshes.
Relative tolerances of $10^{-11}$ and $10^{-10}$ are used for MinRes and CG 
respectively.
}
\footnotesize{
\begin{tabular}{l|lcc||l|lcc}
\hline
\multicolumn{4}{c||}{uniform} & \multicolumn{4}{c}{refined}\\
  \hline
  size  & $\norm{u-u_h}_1$ & \# & $\max_Z|(u_h, z)|$ & size & $\norm{u-u_h}_1$ & \# & $\max_Z|(u_h, z)|$ \\
\hline
14745   & 1.03E-02 (1.14) & 44 & 3.54E-07 & 13080   & 3.11E-02 (0.99) & 50 & 1.68E-07\\
107817  & 4.84E-03 (1.09) & 45 & 2.77E-06 & 98052   & 1.41E-02 (1.14) & 53 & 6.73E-08\\
823881  & 2.36E-03 (1.03) & 45 & 1.38E-06 & 759546  & 6.53E-03 (1.11) & 54 & 8.11E-07\\
6440073 & 1.18E-03 (1.00) & 44 & 1.75E-05 & 5978835 & 3.20E-03 (1.03) & 55 & 2.94E-06\\
\hline
\hline
14739   & 1.14E-02 (1.09) & 21 & 1.30E-03 & 13074   & 5.51E-02 (0.45) & 26 & 6.06E-03\\
107811  & 5.49E-03 (1.06) & 23 & 6.66E-04 & 98046   & 5.05E-02 (0.12) & 27 & 6.32E-03\\
823875  & 2.71E-03 (1.02) & 25 & 3.36E-04 & 759540  & 5.00E-02 (0.02) & 29 & 6.43E-03\\
6440067 & 1.35E-03 (1.00) & 26 & 1.69E-04 & 5978829 & 4.98E-02 (0.01) & 31 & 6.49E-03\\
\hline
\end{tabular}
}
\label{tab:krylov_lagrange_cmp}
\end{center}
\end{table}

From the experiment we conclude that the conjugate gradient method for
\eqref{eq:Axb}, as applied so far, in general does not yield converging 
numerical solutions of \eqref{eq:cg_long}. It is next shown that the issue is
due projector $\la{P}_Z=\la{I}-\la{Z}\transp{\la{Z}}$ which the method uses and 
which is derived from the discrete problem. In particular, we show
that $\la{P}_Z$ is not a correct discretization of a projector used in the
continuous problem \eqref{eq:cg_zperp} (and \eqref{eq:weak}). Following the
continuous problem, a modification to CG is proposed, which leads to a converging 
method.

\subsection{Conjugate gradient method with $\Zpolar$, $\Zperp$ projectors}\label{sec:kpp}
Consider the variational problem \eqref{eq:cg_zperp} which was proven well-posed
in Theorem \ref{thm:cg} under the assumptions $l\in\Zpolar\subset\dual{V}$ and
$u\in\Zperp\subset V$. In this respect, there are two subspaces associated
with \eqref{eq:cg_zperp} and we shall define two projectors
$P:V\rightarrow\Zperp$, $\dual{P}:\dual{V}\rightarrow\Zpolar$ such that for
$u\in V$, $f\in\dual{V}$
\begin{equation}\label{eq:proj}
  \begin{aligned}
    &(P u, v) = (u_{\Zperp}, v)\quad \forall v\in V,\\
    &\brack{\dual{P}f, u} = \brack{f, u-u_z} .
  \end{aligned}
\end{equation}
Similar projectors were discussed in \cite{bochev} for the singular Poisson problem. 
We note that $\brack{f, Pu}=\brack{\dual{P} f, u}$ and thus $\dual{P}$ is the
adjoint of $P$. \deleted{Note also that the two projectors are present in the multiplier formulation \eqref{eq:weak}.}
\begin{lemma} Let $l\in\dual{V}$ and $P, \dual{P}$ be the projectors \eqref{eq:proj}.
  Then $(u, p)\in V\times Z$ solves \eqref{eq:weak} with the right hand side
  $(v, q)\mapsto \brack{l, v}+\brack{0, q}$ if and only if $u\in\Zperp$ and $u$ solves
  \eqref{eq:cg_zperp} with the right hand side $\dual{P}l$.
\end{lemma}
\begin{proof}
It suffices to establish the relation between the right hand sides.
  Testing \eqref{eq:weak} with $(z, 0)$, $z\in Z$ yields that $(p, z)=\brack{l, z}$. In
  turn we have for any $v\in V$
  \[
    \brack{l, v}-(p, v) = \brack{l, v}-(p, v_Z+v_{\Zperp}) = \brack{l, v} -
    \brack{l, v_Z} = \brack{l, v-v_Z} = \brack{l, P v}
  \]
and the new right hand side of \eqref{eq:weak} is therefore
  $(v, q)\mapsto \brack{\dual{P}l, v}+\brack{0, q}$.
\end{proof}

To derive a matrix representation of the projectors with respect to nodal basis 
$V_h=\spn{\set{\phi_i}_{i=1}^n}$, the mappings $\pi_h:V_h\rightarrow\reals^n$ (the
nodal interpolant) and $\mu_h:\dual{V}_h\rightarrow\reals^n$ from
\eqref{eq:pimu} are used. We recall that $(u, v)=\transp{\la{v}}\la{M}\la{u}$ for
$\la{u}=\pi_h u$, $\la{v}=\pi_h v$ and $\la{M}$, $M_{ij}=(\phi_j, \phi_i)$ the mass matrix
while $\brack{f, v}=\transp{\la{f}}\la{v}$ with $\la{f}=\mu_h f$. Finally,
matrix $\la{Y}=\reals^{n\times 6}$ is such that $\la{y}_k=\text{col}_k\la{Y}=\pi_h z_k$ 
where $z_k\in V_h$ belongs to the $L^2$ orthogonal basis of the space of rigid 
motions. Then
\begin{equation}\label{eq:proj_discrete}
  \begin{aligned}
    &\transp{\la{v}}\la{M}\la{P}\la{u} = (Pu, v) = (u, v) - \sum_{k=1}^{6}(u, z_k)(v, z_k) =
  \transp{\la{V}}\la{M}\left(\la{I}-\la{Y}\transp{\la{Y}}\la{M}\right)\la{u},\\
    &\transp{\la{f}}\transp{\dual{\la{P}}}\la{v}=
    \brack{f, Pv} = \brack{f, v} - \sum_{k=1}^{6}\brack{f, z_k}(v, z_k) =
    \transp{\la{f}}\left(\la{I}-\la{Y}\transp{\la{Y}}\la{M}\right)\la{v}
  \end{aligned}
\end{equation}
and $\la{P}=\left(\la{I}-\la{Y}\transp{\la{Y}}\la{M}\right)$ is the representation 
of $P$ while $\dual{P}$ is represented by $\transp{\la{P}}$. We remark that in
addition to $\la{Y}$, the rigid motions $Z_h=\spn{\set{z_k}_{k=1}^6}$ can be
represented in $\reals^n$ by an additional matrix $\la{W}=\la{M}\la{Y}$, which is
$\mu_h$ applied to functionals $v\mapsto(z_k, v)$. Following \cite{kent} the
matrices $\la{Y}$, $\la{W}$ are termed respectively the primal and dual
representation of $Z_h$. Observe that in \eqref{eq:proj_discrete} matrix $\la{P}$ 
uses the primal representation for $\la{u}$ while the vector is expanded in the dual
representation by $\dual{\la{P}}$. Moreover, $L^2$ orthogonality of $Z_h$ yields
$\transp{\la{y}_i}\la{w}_j=\delta_{ij}$. Finally note that the projectors
$\transp{\la{P}}$, $\la{P}$ are implicitly present in the linear system which is 
the discretization of the multiplier problem \eqref{eq:weak} with the orthogonal 
basis of rigid motions
\begin{equation}\label{eq:AAxb}
  \begin{pmatrix}
  \la{A} & \la{W}\\
  \transp{\la{W}} &
  \end{pmatrix}\begin{pmatrix}\la{u}\\
                              \la{p}\end{pmatrix}=
  \begin{pmatrix}\la{b}\\
                 \la{0}\end{pmatrix}.
\end{equation}
Indeed, $\la{p}=\transp{\la{Y}}\la{b}$ from premultiplying the first equation by 
$\transp{\la{Y}}$. Upon substitution the equation reads 
$\la{A}\la{u}=\la{b}-\la{W}\transp{\la{Y}}\la{b}=\transp{\la{P}}\la{b}$. Further
the solution is such that $\la{P}\la{u}=0$.

The situation where the continuous problems \eqref{eq:weak}, \eqref{eq:cg_zperp}
and the discrete problem \eqref{eq:AAxb} use different projectors for the left
and right hand sides contrasts with \eqref{eq:Axb} which utilizes 
$\la{P}_Z$ to obtain consistent right hand side and the solution is such that 
$\la{P}_Z\la{u}=0$ as well. This observation together with the lack of convergence of the
CG method, cf. Table \ref{tab:krylov_lagrange_cmp}, motivate that the CG method on
\eqref{eq:Axb} is used with the following two modifications: (i) the iterations
are started from vector $\transp{\la{P}}\la{b}$, (ii) $\la{P}$ is applied to the
final solution.

The effect of the proposed modifications is shown in Table
\ref{tab:improved_krylov}. The problem from Example \ref{ex:pin_elasticity} is
considered on a non-uniform mesh and CG on \eqref{eq:Axb} is applied with
different combinations of projectors used to obtain the right hand side from 
incompatible vector $\la{b}$ and to orthogonalize the converged solution. 
We observe that only the case $(\transp{\la{P}}, \la{P})$\footnote{
Elements of the tuple denote respectively the projector for the right hand
side and the left hand side.
}
yields optimal convergence. With $(\la{P}_Z, \la{P})$ the rate is slightly smaller
than one. In the remaining two cases the solution do not converge suggesting
that for convergence $\la{P}$ must be applied to the solution.

\begin{table}
  \begin{center}
  \caption{
Convergence of conjugate gradient solutions for \eqref{eq:Axb} with different 
combinations of right hand (horizontal) side and left hand side (vertical)
projectors. The problem from Example \ref{ex:pin_elasticity} is considered.
Preprocessing the right hand side and postprocessing the solution by projectors 
$(\transp{\la{P}}, \la{P})$ yields solutions converging at optimal rate.
  }
\footnotesize{
\begin{tabular}{c|l|ccc||ccc}
\hline
  & \multirow{2}{*}{size} & \multicolumn{3}{c||}{$\la{P}_Z$} & \multicolumn{3}{c}{$\transp{\la{P}}$}\\
  \cline{3-8}
  & & $\norm{u-u_h}_1$ & \# & $\max_Z|(u_h, z)|$ & $\norm{u-u_h}_1$ & \# & $\max_Z|(u_h, z)|$ \\
\hline
  \multirow{4}{*}{\rotatebox[origin=c]{90}{$\la{P}_Z$}}
  & 13074   & 5.51E-02 (0.45) & 26 & 6.06E-03 & 5.53E-02 (0.44) & 27 & 6.05E-03\\
  & 98046   & 5.05E-02 (0.12) & 27 & 6.32E-03 & 5.11E-02 (0.12) & 28 & 6.31E-03\\
  & 759540  & 5.00E-02 (0.02) & 29 & 6.43E-03 & 5.06E-02 (0.01) & 29 & 6.42E-03\\
  & 5978829 & 4.98E-02 (0.01) & 31 & 6.49E-03 & 5.05E-02 (0.00) & 31 & 6.48E-03\\
\hline
\hline
  \multirow{4}{*}{\rotatebox[origin=c]{90}{$\la{P}$}}
  & 13074   & 3.13E-02 (0.98) & 27 & 6.84E-16 & 3.11E-02 (0.99) & 25 & 6.15E-16\\
  & 98046   & 1.45E-02 (1.11) & 28 & 2.94E-14 & 1.41E-02 (1.14) & 27 & 2.92E-14\\
  & 759540  & 6.92E-03 (1.07) & 29 & 6.39E-14 & 6.53E-03 (1.11) & 29 & 6.40E-14\\
  & 5978829 & 3.63E-03 (0.93) & 31 & 2.89E-13 & 3.20E-03 (1.03) & 31 & 2.86E-13\\
\hline
\end{tabular}
}
\label{tab:improved_krylov}
\end{center}
\end{table}

The results shown in Table \ref{tab:improved_krylov} are satisfactory in a sense that 
preprocessing the right hand side with $\transp{\la{P}}$ and postprocessing
the solution with $\la{P}$ improved the convergence properties of the CG method for
\eqref{eq:Axb}. However, the modifications alter the original discrete problem 
and thus the properties of the new problem should be discussed. We note that in
the discussion $\la{Z}$, $\la{Y}$ are respectively $\la{I}$ and $\la{M}$ orthogonal
basis of the nullspace of $\la{A}$. Further, the transformation matrix between
the basis is $c\in\reals^{6\times 6}$ such that $\la{Z}=\la{Y}c$ and we have 
$\transp{\la{Y}}\la{M}\la{Z}=c$.

First, admissibility of the modified right hand side $\transp{\la{P}}\la{b}$ 
is considered. Using the transformation matrix it holds that $\transp{\la{Z}}\transp{\la{P}}\la{b}=0$ 
and thus $\transp{\la{P}}\la{b}$ is compatible and the solution can be 
obtained by a pseudoinverse (or equivalently by CG). The computed solution of the new linear 
system then satisfies $\transp{\la{Z}}\la{u}=0$. However, the continuous problem
requires orthogonality $\transp{\la{Y}}\la{M}\la{u}=Ch$. As the two conditions 
are related through
$
  \semi{\transp{\la{Y}}\la{M}\la{u}}^2
  = 
  \transp{\la{u}}\la{MZ}\inv{(\transp{c}c)}\transp{\la{Z}}\la{M}\la{u}
  =
  \transp{\la{u}}\la{MZ}\inv{(\transp{\la{Z}}\la{M}\la{Z})}\transp{\la{Z}}\la{M}\la{u},
$
and $\transp{\la{Z}}\la{Z}=\la{I}$, orthogonality in the $L^2$ inner product 
depends on similarity of the mass matrix with identity. This is essentially a
condition on the mesh and $\semi{\transp{\la{Y}}\la{M}\la{Z}}\geq C$ is possible 
(as observed in Table \ref{tab:improved_krylov}).

To enforce orthogonality constraint $\transp{\la{Y}}\la{M}\la{u}=0$ without
postprocessing we shall finally consider linear system $\la{Au}=\transp{\la{P}}\la{b}$
and require $\la{P}\la{u}=0$ for uniqueness. In this case the solution is not
provided by pseudoinverse $\la{B}_A$. However, a similar construction based on
the generalized eigenvalue problem can be used instead.

\begin{lemma} Let $\la{u}$ be a unique solution of $\la{Au}=\transp{\la{P}}\la{b}$, 
  satisfying $\la{P}\la{u}=0$ and $\la{\Gamma}\in\reals^{n\times n}$, 
  $\la{U}\in\reals^{n\times n-6}$ such that 
  $\la{A}\la{U} = \la{M}\la{U}\la{\Gamma}$, $\transp{\la{U}}\la{M}\la{U}=\la{I}$. Then 
  $\la{u}=\la{B}\transp{\la{P}}\la{b}$ where $\la{B}=\la{U}\inv{\la{\Gamma}}\transp{\la{U}}$.
\end{lemma}
\begin{proof}
First, note that the existence of matrices $\la{U}$, $\la{\Gamma}$ follows from
  positive semi-definiteness of $\la{A}$. Further, by $\la{M}$ orthogonality of the eigenvectors
$\la{M}\la{U}\la{x}=\transp{\la{P}}\la{b}$ holds with $\la{x}=\transp{\la{U}}\la{b}$.
  As $\transp{\la{Y}}\la{M}\la{U}=0$ any vector $\la{B}\la{b}$ is $\la{M}$ orthogonal with 
  $\la{Y}$ and thus $\la{P}\la{B}\la{b}=0$. It remains to show that the
  composition $\la{A}\la{B}$ is the identity on the subspace spanned by columns
  of $\la{M}\la{U}$
  \[
  \la{A}\la{B}\la{M}\la{U}
  =
  \la{A}  \la{U}\inv{\la{\Gamma}}\transp{\la{U}} \la{M}\la{U}
  =
  \la{A}  \la{U} \inv{\la{\Gamma}}
  =
  \la{M}\la{U}\la{\Gamma} \inv{\la{\Gamma}}
  =
  \la{M}\la{U}.
  \]
\end{proof}

\section{Natural norm formulation}\label{sec:energy}

An attractive feature of the variational problem \eqref{eq:cg_long} is the fact
that the resulting linear system is amenable to solution by the CG method, which
when modified following \S \ref{sec:krylov} yields converging solutions.
However, the projectors $\dual{P}$, ${P}$ are only applied as pre and
postprocessor and the CG loop is in this respect detached from
the continuous problem. Moreover the method requires a special preconditioner that
handles the nullspace of matrix $\la{A}$. A formulation which leads to a
positive definite linear system requiring only a regular (not nullspace aware)
preconditioner shall be studied next.

\begin{theorem} Let $a:V\times V\rightarrow\reals$, 
  $a(u, v)=2\mu(\epsilon(u), \epsilon(v))+\lambda(\nabla\cdot u, \nabla\cdot v)$
  and let $l\in\Zpolar$. There exists a unique $u\in V$ satisfying 
  \begin{equation}\label{eq:energy}
    a(u, v) + (u_Z, v_Z) = \brack{l, v} \quad \forall v\in V.
  \end{equation}
  Moreover $u\in\Zperp$.
\end{theorem}
\begin{proof} Recall that the bilinear form above is the inner product $(u, v)_E$
  from \eqref{eq:Vinner} which induces an equivalent norm on $V$, cf. Lemma \ref{lm:inners}. 
  The existence and uniqueness of the solution now follow from the Lax-Milgram lemma.
  Testing the equation with $v=z\in Z$ yields $(u, z)=0$ and in turn
  $u\in\Zperp$.
\end{proof}
We remark that the solution of \eqref{eq:energy} and \eqref{eq:cg_zperp} are
equivalent because $l\in\Zpolar$. Note also that Theorem \ref{thm:m_norm} gives equivalence bounds
$\inv{(1+C)}\norm{u}^2_M\leq \norm{u}^2_E \leq \norm{u}^2_M$ for all $u\in V$ and
$C=C(\Omega)$. In turn the Riesz map with respect to the inner product 
$(u, v)_M=a(u, v)+(u, v)$ defines a suitable $h$ robust preconditioner for
\eqref{eq:energy}. Finally, observe that the $L^2$ orthogonality of decomposition 
$u=u_Z+u_{\Zperp}$ is respected by the inner product $(\cdot,
\cdot)_E$, see \eqref{eq:Vinner}. The norm $\norm{u}_E$, see \eqref{eq:Vnorm}, thus 
considers $Z$ and $\Zperp$ with $L^2$ norm and $a$ induced norm which are the natural 
norms for the spaces.

Using \eqref{eq:proj_discrete} the natural norm formulation \eqref{eq:energy}
leads to a positive definite linear system
\[
  \left[\la{A} + \la{M}\la{Y}\transp{\left(\la{M}\la{Y}\right)}\right]\la{u}
  =
  \transp{\la{P}}\la{b}.
\]
where we recognize a dense matrix from the discretization of $\mathcal{B}_E$ 
preconditioner of the Lagrange multiplier formulation, cf. Theorem \ref{thm:e_norm}. 
Therein the inverse of the matrix was of interest. However, relevant for the CG 
method here is only the matrix vector product, which can be computed efficiently 
by storing separately $\la{A}$ and $\la{MY}$, the dual representation of rigid 
motions in $V_h$.

With \eqref{eq:energy} we finally revisit the test problem from Example
\ref{ex:pin_elasticity}. Results of the method are summarized in Table
\ref{tab:energy}. Optimal convergence rate is observed with both uniform and
nonuniform meshes. In the uniform case CG iteration count with the proposed Riesz map
preconditioner approximated by $\text{AMG}(\la{A}+\la{M})$ remains bounded. There is
a slight growth in the refined case. An interesting observation is the fact that the 
error in the orthogonality constraint is smaller in comparison to the Lagrange multiplier 
formulation, cf. Table \ref{tab:krylov_lagrange_cmp}.

\begin{table}
  \begin{center}
  \caption{
    Convergence study of the natural norm formulation \eqref{eq:energy} for the
    singular elasticity problem from Example \ref{ex:pin_elasticity}. The system
    is solved with relative tolerance $10^{-11}$. The CG method uses
    preconditioner $\text{AMG}(\la{A}+\la{M})$. Iterations count are bounded in 
    the uniform case while a slight growth can be seen in the refined one.
    The solutions converge at optimal rate.
  }
\footnotesize{
\begin{tabular}{l|lcc||l|lcc}
\hline
  \multicolumn{4}{c||}{uniform} & \multicolumn{4}{c}{refined}\\
  \hline
  size  & $\norm{u-u_h}_1$ & \# & $\max_Z|(u_h, z)|$ & size & $\norm{u-u_h}_1$ & \# & $\max_Z|(u_h, z)|$ \\
\hline
14739   & 1.03E-02 (1.14) & 33 & 2.57E-08 & 13074   & 3.11E-02 (0.99) & 39 & 3.70E-08\\
107811  & 4.84E-03 (1.09) & 29 & 1.80E-05 & 98046   & 1.41E-02 (1.14) & 41 & 3.46E-08\\
823875  & 2.36E-03 (1.03) & 37 & 9.23E-09 & 759540  & 6.53E-03 (1.11) & 43 & 8.90E-08\\
6440067 & 1.18E-03 (1.00) & 33 & 2.38E-05 & 5978829 & 3.20E-03 (1.03) & 46 & 3.53E-08\\
\hline
\end{tabular}
}
\label{tab:energy}
\end{center}
\end{table}

\section{Nearly incompressible materials}\label{sec:mixed}
So far we have assumed that $\mu$ and $\lambda$ are comparable in magnitude. 
In this section we handle the case where $\lambda \gg \mu$ and the 
material is nearly incompressible. The variational problems \eqref{eq:weak_long}, 
\eqref{eq:cg_long}, \eqref{eq:energy} studied thus far were based on the pure displacement 
formulation of linear elasticity \eqref{eq:strong} and $H^1$ conforming finite 
element spaces were used for their discretization. Due to the \textit{locking} 
phenomenon the approximation properties of their respected solutions are known 
to degrade for nearly incompressible materials with $\lambda\gg\mu$, 
(equivalently Poisson ratio close to 1/2), see e.g. \cite[ch 6.3]{braess}. 
Moreover, the incompressible limit presents a difficulty for convergence of 
iterative methods in the standard form.  

Methods robust with respect to increasing $\lambda$ can be 
formulated using a discretization with nonconforming elements, \cite[ch 11.4]{brenner}. 
However, this method fails to satisfy the Korn's inequality. To the authors'
knowledge the only 
primal 
conforming finite element method that 
is both robust in $\lambda$ and satisfies Korn's inequality is \cite{mardal2002robust, mardal2006observation}. 
\added[id=rev1, remark={comment 3}]{
In addition to problems with the discretization, standard multigrid algorithms do not work
well for large $\lambda$ and special purpose algorithms must be used \cite{schoberl1999multigrid}.   
Related discontinuous Galerkin formulation based on H(div)-conforming elements are descibed in 
\cite{hong2016robust} where also a H(div) multigrid method is introduced.
}
For this reason we resort to a more straightforward solution of the mixed formulation
where an additional variable, the \textit{solid pressure} $p$, is introduced. 
Let the solid pressure be defined as $p=\lambda \Div u$ so that
\eqref{eq:weak_long} is reformulated as
\begin{equation}\label{eq:mixed_strong}
  \begin{aligned}
 &\Div{(2\mu\epsilon(u))} - \nabla p = f&\quad\text{ in }\Omega,\\
 &\lambda\Div u - p                  = 0&\quad\text{ in }\Omega,\\
 &\sigma(u)\cdot n = h                   &\mbox{ on }\partial\Omega.\\
  \end{aligned}
\end{equation}
Note that the problem is singular, since any pair
$z\in Z$, $p=0$ can be added to the solution. In fact such pairs constitute the
kernel of \eqref{eq:mixed_strong}. To obtain a unique solution we shall as in \S
\ref{sec:lagrange}, require that $u$ is orthogonal to the rigid motions $Z$.

Setting $Q=L^2(\Omega)$ we shall consider a variational problem
for triplet $u\in V$, $p\in Q$, $\nu\in Z$ such that
\begin{equation}\label{eq:robust_weak}
  \begin{aligned}
    &2\mu(\epsilon(u), \epsilon(v)) + (p, \Div{v}) + (\nu, v) = \brack{l, v}\quad&\forall v\in V,\\
&(q, \Div{u}) - \lambda^{-1}(p, q)                           = 0 \quad&\forall q\in Q,\\
    &(\eta, u)                                                 = 0 \quad&\forall \eta \in Z.\\
  \end{aligned}
\end{equation}
Equation \eqref{eq:robust_weak} defines a double saddle point problem
\[
  \mathcal{A}
  \begin{pmatrix}u\\ p\\ \nu\end{pmatrix}
  =
  \begin{pmatrix}
       A &              B &           D\\
\dual{B} & -\lambda^{-1}C & \phantom{0}\\
\dual{D} & \phantom{0}    & \phantom{0}
  \end{pmatrix}
  \begin{pmatrix}u\\ p\\ \nu\end{pmatrix}
  =
  \begin{pmatrix}l \\ 0 \\ 0\end{pmatrix}
\]
with operators $A:V\rightarrow\dual{V}$, $B:Q\rightarrow\dual{V}$, 
$C:Q\rightarrow\dual{Q}$, $D:Z\rightarrow\dual{V}$ and functional $l: V\rightarrow\mathbb{R}$ defined as
\begin{equation}\label{eq:op_def}
\begin{aligned}
&\brack{Au, v}=2\mu(\epsilon(u), \epsilon(v)), &\quad\quad\brack{Bp, v}=(p, \nabla\cdot v),\\
  &\brack{Cp, q}=(p, q),                       &\quad\quad\brack{D\eta, v}=(\eta, v)
\end{aligned}
\end{equation}
and
\begin{equation}\label{eq:l_def}
\brack{l, v} = (f, v) + (h, v)_{\partial\Omega}.
\end{equation}

To show well-posedness of the constrained mixed formulation \eqref{eq:robust_weak} 
the abstract theory for saddle points problems with small (note that that $\lambda\gg 1$) 
penalty terms \cite[ch 3.4]{braess} is applied. To this end we introduce
the bilinear forms $a(u, v)=\brack{Au, v}$,
\begin{equation}\label{eq:pen_infsup}
  b(v, (p, \eta))=\brack{Bp, v}+\brack{D\eta, v},
\end{equation}
$c((p, \eta), (q, \eta))=\brack{Cp, q}$ so that \eqref{eq:robust_weak} is recast
as: Find $u\in V$, $(p, \nu)\in Q\times Z$ satisfying
\begin{equation}\label{eq:penalty}
  \begin{aligned}
    &a(u, v) + b(v, (p, \nu))            = \brack{l, v} \quad&\forall v\in V,\\
    &b(u, (q, \eta)) - \lambda^{-1}(p, q) = 0               \quad&\forall (q, \eta)\in Q\times Z.
  \end{aligned}
\end{equation}
The space $Q\times Z$ will be considered with the norm 
$\norm{(p, \eta)}=\sqrt{\norm{p}^2+\norm{\eta}^2}$, while $V$ is considered with the $H^1$ norm.
Following \cite[thm 4.11]{braess} the problem \eqref{eq:penalty} is well-posed
provided that the assumptions of Brezzi theory hold and in addition $c$ is
continuous and $c$ and $a$ are positive
\[
a(u, u)\geq 0 \quad\forall u\in V,
\quad\quad
\text{and}
\quad\quad
c((p, \eta), (p, \eta))\geq 0 \quad\forall (p, \eta)\in Q\times Z.
\]

We review that continuity and $V$-ellipticity of $a$ on $\Zperp$ was shown in 
Theorem \ref{thm:h1_norm} and as $a(z, z)=0$, $z\in Z$, the form is positive on $V$. 
Moreover, by Lemma \ref{lm:help} and Cauchy-Schwarz 
\[
  \begin{split}
    b(v, (p, \eta)) = (p, \Div v)+(v, \eta) 
    \leq \sqrt{3}\norm{p}\norm{\nabla v} + \norm{v}\norm{\eta}
    &\leq \sqrt{3}\sqrt{\norm{v}^2+\norm{\nabla v}^2}\sqrt{\norm{p}^2+\norm{\eta}^2}\\
    &\leq \beta^{*}\norm{v}_1\norm{(p, \eta)}
  \end{split}
\]
holds for any $v\in V$, $(p, \eta)\in Q\times Z$.
It is easy to observe that continuity and positivity of the bilinear form $c$ 
hold and thus \eqref{eq:penalty} is well-posed provided that the inf-sup
condition is satisfied. We note that the proof requires extra regularity of the
boundary.

\begin{lemma}\label{lm:penalty}
  Let $\Omega$ with a smooth boundary and $b$ be the bilinear form over 
  $V\times(Q\times Z)$ defined in
  \eqref{eq:pen_infsup}. There exists $\beta_*=\beta_*(\Omega)$ such that
  \[
    \sup_{v\in V}\frac{b(v, (p, \eta))}{\norm{v}_1}\geq\beta_*\norm{(p, \eta)}
    \quad\forall (p, \eta)\in Q\times Z.
  \]
\end{lemma}
\begin{proof}
  Let $p\in Q$ and $\eta\in Z$ given. Following \cite[thm 11.2.3]{brenner} there
  exists for every $p$ a $v^{*}\in V$ such that
  \begin{subequations}
    \begin{align}
      &p=\Div{v^{*}},\label{eq:help_div}\\
      &\norm{v^{*}}_1\leq C(\Omega)\norm{p}.\label{eq:help_bound}
    \end{align}
  \end{subequations}
  The element $v^{*}$ is constructed from the unique solution of the Poisson
  problem
  \begin{equation}\label{eq:help_poisson}
    \begin{aligned}
    -&\Delta w = p&\quad\text{ in }\Omega,\\
          &w = 0&\quad\text{ on }\partial\Omega,
    \end{aligned}
  \end{equation}
  taking $v^*=-\nabla w$. Observe that the computed $v^{*}\in \Zperp$
  \begin{equation}\label{eq:help_og}
  -(z, v^*) = \int_{\Omega}z\nabla w = 
  \int_{\partial\Omega}w z\cdot n - \int_{\Omega}w\Div z= 0 \quad \forall z\in Z.
  \end{equation}
Orthogonality of $v^*$ and \eqref{eq:help_div} yields that $b(v^*+\eta, (p,
  \eta))=(p, \Div v^*)+(\eta, \eta)=\norm{p}^2+\norm{\eta}^2$.
  Further, by Cauchy-Schwarz and Young's inequalities and Lemma \ref{lm:help}
  \[
  \begin{split}
    \norm{v^*+\eta}^2_1 &= \norm{v^*+\eta}^2 + \norm{\nabla(v^*+\eta)}^2\\
     &= \norm{v^*}^2 + \norm{\eta}^2 + \norm{\nabla v^*}^2 + 2(\nabla v^*, \nabla
     \eta)+\norm{\nabla \eta}^2\\
                          &\leq 2\norm{v^*}^2_1 + 2(\norm{\eta}^2+\norm{\nabla \eta}^2)
                          \leq 2\norm{v^*}^2_1 + 2 C(\Omega)\norm{\eta}^2
  \end{split}
  \]
so that $\norm{v^*+\eta}_1 \leq c(\Omega)\norm{(p, \eta)}$. Combining the observations
\[
  \sup_{v\in V}\frac{b(v, (p, \eta))}{\norm{v}_1}
  \geq
  \frac{b(v^*+\eta, (p, \eta))}{\norm{v^* + \eta}_1}
  =
  \frac{\norm{p}^2+\norm{\eta}^2}{\norm{v^* + \eta}_1}
  \geq\frac{1}{c}\sqrt{\norm{p}^2+\norm{\eta}^2}=\frac{1}{c}\norm{(p, \eta)}.
\]
\end{proof}
We remark that none of the constants of the problem \eqref{eq:penalty} depends on 
$\lambda$ despite the norm of $Q\times Z$ being free of the parameter, cf. also
\cite{klawonn_block, klawonn_optimal}. Observe also that with $H^1$ norm on $V$ the
boundedness constant of $a$ depends on $\mu$, cf. Theorem \ref{thm:h1_norm}, and
thus the parameter shall be included in the norm to get a $\mu$ independent
preconditioner. This choice corresponds to considering the space $V$ with the norm 
$u\mapsto\sqrt{2\mu\norm{\epsilon(u)}^2+\norm{u}^2}$.

Motivated by the above, we shall consider as the preconditioner for the
well-posed problem \eqref{eq:penalty} a Riesz map 
$\mathcal{B}:\dual{(V\times Q\times Z)}\rightarrow (V\times Q\times Z)$ with 
respect to the inner product inducing the norm 
$(u, p, \eta)\mapsto\sqrt{2\mu\norm{\epsilon(u)}^2+\norm{u}^2+\norm{p}^2+\norm{\eta}^2}$
\begin{equation}\label{eq:mixed_precond}
  \mathcal{B} = 
  \inv{
  \begin{pmatrix}
    A+M &    & \\
        & C  & \\
        &    & I
  \end{pmatrix}
  },
\end{equation}
where $M$, $I$ were defined respectively in \eqref{eq:Vinner} and \eqref{eq:B1}. 
Similar preconditioners for the Dirichlet problem has been discussed in \cite{klawonn_optimal,
fortin2010iterative}.

\begin{remark}[Lemma \ref{lm:penalty} in the discrete case]\label{rmrk:only}
The continuous inf-sup condition can be extended to Taylor-Hood discretizations
  in the following way.
We consider $V_h\subset V$, $Q_h\subset Q$ approximated with the lowest order 
Taylor-Hood element. Given $p_h\in Q_h$ both the element $v^{*}_h\in V_h$ and 
$w_h\in Q_h$ from Lemma \ref{lm:penalty} are found as the solution to the mixed 
  Poisson problem
\[
\begin{aligned}
  &(v_h^{*}, v) + (\nabla_h w_h, v) = 0&\quad \forall v\in V_h,\\
  &(\nabla_h q, v^{*}_h)            = -(p_h, q) &\quad \forall q\in Q_h.
\end{aligned}
\]
The problem is well-posed due to the weak inf-sup condition
\[
\sup_{v_h\in V_h}\frac{(v_h, \nabla_h q_h)}{\norm{v_h}} \geq C \norm{q_h}_1
  \quad \forall q_h\in Q_h.
\]
Since $z\in V_h$ a direct calculation shows that the orthogonality condition \eqref{eq:help_og} 
is satisfied. 

Both in the above and in the construction of the proof of Lemma \ref{lm:penalty} 
  we relied on a well-posed mixed Poisson problem to obtain orthogonality with 
  respect to the kernel. We note that stable Stokes element $P_2-P_0$ does not 
  allow for such a construction and does not give $h$ uniform bounds.
\end{remark}

To show that the preconditioner \eqref{eq:mixed_precond} is robust with respect to 
$\lambda$, we consider \eqref{eq:robust_weak} with $\mu=1$ and data $h=0$ and $f=u^{*}$
defined in Example \ref{ex:pin_elasticity} while the value of $\lambda$ varies 
in the interval $\left[1,10^{15}\right]$. Moreover, an exactly incompressible case
shall be considered, where the operator $C$ is set to zero.

The spaces $V$ and $Q$ are approximated by lowest order Taylor-Hood elements
for which the discrete inf-sup condition from Lemma \ref{lm:penalty} holds
following Remark \ref{rmrk:only}.

\added[id=rev1, remark={comment 3}]{As with the previous experiments the
  approximate inverse of $A+M$ and $C$ blocks are realized by single multigrid
  $V$ cycle.} The final block corresponding to $Z$ is an 
identity due to the employed orthonormal basis.

The system is solved using the MinRes 
method and absolute tolerance $10^{-8}$ for the preconditioned residual as a
convergence criterion.

From the results of the experiment, summarized in Table \ref{tab:mixed_dx}, it
is evident that the iteration count is bounded in $\lambda$ as well as in the 
discretization parameter. We note that the error in the orthogonality constraint
is comparable to that reported in Table \ref{tab:krylov_lagrange_cmp} for the
Lagrange multiplier formulation of the pure displacement problem.


\begin{table}[ht!]
\setlength\extrarowheight{1.0pt}
\begin{center}
\caption{Iteration counts of the preconditioned MinRes method for mixed linear
  elasticity problem \eqref{eq:robust_weak} and different values of Lam\'{e}
  constant $\lambda$. The exact incompressibility case is denoted by
  $\lambda=\infty$. The iteration counts remain bounded for the considered
  values of the parameter.
}
\footnotesize{
\begin{tabular}{ll|cccccc|c}
\hline
\multirow{2}{*}{$\text{dim}(V)$} & \multirow{2}{*}{$\text{dim}(Q)$} 
                                 & \multicolumn{6}{c|}{$\lambda$}
                                 & \multirow{2}{*}{$\max_{Z, \lambda}\semi{(u_h,z)}$}\\
  \cline{3-8}
                   && $10^{0}$ & $10^{4}$ & $10^{8}$ & $10^{12}$ & $10^{15}$ & $\infty$ &\\
\hline
14739   & 729 & 81 & 87 & 88 & 87 & 88 & 90& 9.56E-07 \\
107811  & 4913 & 78 & 77 & 80 & 79 & 82 & 79& 3.66E-06 \\
823875  & 35937 & 69 & 72 & 72 & 72 & 72 & 72& 4.02E-05 \\
6440067 & 274625 & 67 & 66 & 66 & 66 & 67 & 65& 6.68E-05 \\
\hline
\end{tabular}
}
\label{tab:mixed_dx}
\end{center}
\end{table}

\subsection{\added[id=rev1, remark={comment5}]{
    Single saddle point formulation}
}\label{sec:mixed_pcg}

Using formulation \eqref{eq:penalty} the weak solution of \eqref{eq:mixed_strong} is
computed from a double saddle point problem. However, if considered in $\Zperp\times Q$
the mixed formulation of linear elasticity has just a single saddle point.
A formulation which preserves this property is pursued next.

We begin by observing a few properties of the solution of the double saddle point problem.
\begin{remark}[Properties of solution of \eqref{eq:penalty}]\label{rk:pcg}
(i) In the solution triplet $u\in V$, $p\in Q$, $\nu\in Z$
the rigid motion satisfies $(\nu, z)=\brack{l, z}$ for all $z\in Z$. In particular
$\nu=0$ if and only if $l\in\Zpolar$.
(ii) The triplet $u, p, \nu$ solves \eqref{eq:penalty} if and only if $u, p, 0$ satisfies 
\eqref{eq:penalty} with $l\in\Zpolar$. 
\end{remark}
We note that the first property follows by testing \eqref{eq:penalty}
with $v\in Z$, $p=0$, $\eta=0$ while the second is readily checked
by direct calculation. Note also that if orthonormal basis of the space
of rigid motions is employed the Lagrange multiplier in \eqref{eq:penalty}
is computed simply by evaluating the right hand side.

Due to Remark \ref{rk:pcg} it is only $u\in V$ and $p\in Q$ which are
the non-trivial unknowns of the double saddle point problem \eqref{eq:penalty}.
The pair can be obtained also as a solution of a system with single saddle point. 

\begin{theorem}\label{thm:mixed_single}
  Let $A: V\rightarrow\dual{V}$, $B: Q\rightarrow\dual{V}$, $C: Q\rightarrow\dual{Q}$
  be the operators defined in \eqref{eq:op_def} and $Y: V\rightarrow\dual{V}$ be such
  that $\brack{Yu, v}=(u_Z, v_Z)$ where $V\ni u=u_Z + u_{\Zperp}$ and $u_Z\in Z$, $u_{\Zperp}\in \Zperp$. 
  Then for each $l\in \dual{V}$ there exists unique $u\in V$, $p\in Q$ such that
  \begin{equation}\label{eq:mixed_pcg}
    \mathcal{A}
    \begin{pmatrix}
      u\\
      p\\
    \end{pmatrix}
    =    
    \begin{pmatrix}
      A + Y    & B\\
      \dual{B} & -\lambda^{-1}C 
    \end{pmatrix}
    \begin{pmatrix}
      u\\
      p\\
    \end{pmatrix}
    =
    \begin{pmatrix}
      l\\
      0
    \end{pmatrix}.
  \end{equation}
  Moreover, if $l\in\Zpolar$ then $u\in\Zperp$ and the triplet
  $u, p, 0$ is the unique solution of \eqref{eq:penalty} with the right hand side $l$.
\end{theorem}
\begin{proof}
  We apply the results of \cite[ch 3.4]{braess} for the abstract saddle point
  systems with penalty terms. To this end we observe that operators $A+Y$, $B$, $C$ are
  clearly bounded on the respected spaces, while $C$ is coercive on $Q$.
  The inf-sup condition for $B$ can be verified as in the proof of Lemma \ref{lm:penalty}. Indeed,
  let $v^{*}\in V$ be the element constructed in \eqref{eq:help_poisson}. Then
  by \eqref{eq:help_div} and \eqref{eq:help_bound}
  \[
    \sup_{v\in V}\frac{\brack{B p, v}}{\norm{v}_1}
    =
    \sup_{v\in V}\frac{(p, \nabla\cdot v)}{\norm{v}_1}
    \geq
    \frac{(p, \nabla\cdot v^{*})}{\norm{v^{*}}_1}
    \geq\frac{1}{C(\Omega)}\norm{p}.
  \]
  Let next $C_1$ be the constant from Korn's inequality \eqref{eq:kornZ}
  while $C_2$ should denote the constant from inequality \eqref{eq:help3}.
  Using decomposition $u=u_{Z}+u_{\Zperp}$ and $\norm{u}^2_1\leq 2(\norm{u_{\Zperp}}^2_1+\norm{u_Z}^2_1)$
  the coercivity of $A+Y$ on $V$ now follows
  \[
    \begin{split}
  \brack{(A+Y) u, u}&=2\mu(\epsilon(u), \epsilon(u)) + (u_Z, u_Z)=
  2\mu(\epsilon(u_{\Zperp}), \epsilon(u_{\Zperp})) + (u_Z, u_Z)\\
  &\geq
  2\mu C_1\norm{u_{\Zperp}}^2_1 + C_2^{-2}\norm{u_{z}}_1^2
  \geq\frac{1}{2}\min{(2\mu C_1, C_2^{-2})}\norm{u}^2_1.
  \end{split}
\]

To verify that $u\in \Zperp$ the equation \eqref{eq:mixed_pcg} is applied
to pair $z, 0$, where $z\in Z$ is arbitrary, yielding $(u_Z, z)=\brack{l, z}=0$
as $l$ is in the polar set of $Z$. From $u\in \Zperp$ it follows that the last
equation in \eqref{eq:robust_weak} holds while with $\nu=0$ the first two
equations to be satisfied by $u, p$ are precisely \eqref{eq:mixed_pcg}.
This verifies the final statement from the theorem.
\end{proof}

Using equivalence of norms shown in Lemma \ref{lm:inners} and operator
preconditioning the preconditioner for the well-posed problem \eqref{eq:mixed_pcg}
is chosen as
\begin{equation}
  \mathcal{B}=
    \inv{
  \begin{pmatrix}
    A+M &    \\
        & C  \\
   \end{pmatrix}
   }
\end{equation}
with $M$ defined in \eqref{eq:Vinner}.

To show that $\mathcal{B}$ defines a parameter robust preconditioner for $\mathcal{A}$
we reuse the experimental setup from the previous section, that is, 
we consider \eqref{eq:robust_weak} with $\mu=1$, $h=0$ and $f=u^{*}$ (see
Example \ref{ex:pin_elasticity}) and $\lambda$  drawn from the interval
$\left[1,10^{15}\right]$. The operators are discretized with the $P_2-P_1$
Taylor-Hood elements which are stable for the problem following Lemma \ref{lm:penalty}.
Note that discretization of operator $A+Y$ in \eqref{eq:mixed_pcg} leads to a dense
matrix, however, similar to \S \ref{sec:energy}, its assembly is not needed
to compute the action. As with the double saddle point problem the action of
the discrete preconditioner is computed with algebraic multigrid while
the system is solved with MinRes method and absolute tolerance $10^{-8}$
for the preconditioned residual norm. We remark that the iterative solver
uses a right hand side orthogonalized with the discrete projector $\transp{\la{P}}$
from \eqref{eq:proj_discrete}, cf. Theorem \ref{thm:mixed_single}.

The results of the experiment are summarized in Table \ref{tab:mixed_dx_pcg}.
We observe that with the proposed preconditioner the iterations are bounded
both in $\lambda$ and the discretization parameter. The table also lists
the error in the orthogonality constraint $(u_h, z)=0\,\forall z\in Z$. With
the chosen convergence criterion the error is about factor 10 larger than for the
double saddle point formulation, cf. Table \ref{tab:mixed_dx}, while on the
finer meshes fewer iterations of the current solver are required for convergence.

\begin{table}[ht!]
\setlength\extrarowheight{1.0pt}
\begin{center}
\caption{Iteration counts of the preconditioned MinRes method for mixed linear
  elasticity problem \eqref{eq:mixed_pcg} and different values of Lam\'{e}
  constant $\lambda$. The exact incompressibility case is denoted by
  $\lambda=\infty$. The iteration counts remain bounded for the considered
  values of the parameter.
}
\footnotesize{
\begin{tabular}{ll|cccccc|c}
\hline
\multirow{2}{*}{$\text{dim}(V)$} & \multirow{2}{*}{$\text{dim}(Q)$} 
                                 & \multicolumn{6}{c|}{$\lambda$}
                                 & \multirow{2}{*}{$\max_{Z, \lambda}\semi{(u_h,z)}$}\\
  \cline{3-8}
                   && $10^{0}$ & $10^{4}$ & $10^{8}$ & $10^{12}$ & $10^{15}$ & $\infty$ &\\
\hline
14739   & 729    & 80 & 89 & 103 & 97 & 97 & 104 & 2.59E-05 \\
107811  & 4913   & 60 & 91 & 94 & 93 & 93 & 92   & 8.79E-05 \\
823875  & 35937  & 48 & 66 & 75 & 69 & 71 & 66   & 4.42E-04 \\
6440067 & 274625 & 36 & 49 & 50 & 52 & 50 & 50   & 5.35E-04 \\
\hline
\end{tabular}
}
\label{tab:mixed_dx_pcg}
\end{center}
\end{table}

\section{Conclusions}\label{sec:fin}
We have studied the singular Neumann problem of linear elasticity. Five
different formulations of the problem have been analyzed and mesh independent
preconditioners established for the resulting linear systems within 
the framework of operator preconditioning. We have proposed a preconditioner 
for the (singular) mixed formulation of linear elasticity, that is robust with 
respect to the material parameters. Using an orthonormal basis 
of the space of rigid motions, discrete projection operators have been derived 
and employed in a modification to the conjugate gradient method to ensure 
optimal error convergence of the solution.  
\appendix\section{Eigenvalue bounds for Lagrange multiplier preconditioners}\label{sec:appendix}
Bounds for the eigenvalues of operators $\mathcal{B}_E\mathcal{A}$ and
$\mathcal{B}_M\mathcal{A}$ from \eqref{eq:weak} and \eqref{eq:BEM} are 
approximated by considering the eigenvalue
problems
\begin{equation}\label{eq:eigenvalues}
  \begin{pmatrix}
  \la{A}          & \la{B}\\
 \transp{\la{B}}  &
  \end{pmatrix}\begin{pmatrix}\la{u}\\
                            \la{p}\end{pmatrix}=
                            \lambda
  \inv{\la{B_i}}\begin{pmatrix}\la{u}\\
                            \la{p}\end{pmatrix}
\end{equation}
with the left hand side the discretization of \eqref{eq:weak} and $\la{B_i}$,
$i\in\set{E, M}$ discretizations of preconditioners $\mathcal{B}_i$ from
\eqref{eq:BEM}. The spectrum of the symmetric, indefinite problem \eqref{eq:eigenvalues} 
is a union of negative and positive intervals $[\lambda^{-}_{\text{min}}, 
\lambda^{-}_{\text{max}}]$, $[\lambda^{+}_{\text{min}}, \lambda^{+}_{\text{max}}]$. 
Following the analysis in Theorems \ref{thm:e_norm} and \ref{thm:m_norm} negative 
bounds equal to -1 are expected for both preconditioners. Further, the positive 
eigenvalues are bounded from above by 1. 
Finally, $\lambda^{+}_{\text{min}}=-1$ for $\mathcal{B}_E$ while the constant $C=C(\Omega)$ 
from the Korn's inequality determines the bound for $\mathcal{B}_M$.

In the experiment, $\Omega$ as a cube from Example \ref{ex:pin_elasticity} and a
hollow cylinder with inner and outer radii $\tfrac{1}{2}$, $1$ and height $2$ are 
considered. Lam\'{e} constants $\mu=384$, $\lambda=577$ are used. For both bodies 
$C\approx 1$ is observed, cf. Table \ref{tab:eigenvalues}. The remaining bounds 
agree well with the analysis.

\begin{table}[ht!]
  \begin{center}
  \caption{
    Spectral bounds for eigenvalue problems \eqref{eq:eigenvalues}. (Top) The
    body is cube. (Bottom) The body is a cylinder.
  }
\footnotesize{
\begin{tabular}{c|l|ccccc}
\hline
  &  size & $\kappa$ & $\lambda^{-}_{\text{min}}+1$ & $\lambda^{-}_{\text{max}}+1$ &
                    $\lambda^{+}_{\text{min}}-1$ & $\lambda^{+}_{\text{max}}-1$\\
\hline
  \multirow{4}{*}{\rotatebox[origin=c]{90}{$\la{B}_E$}} & 87    & 1.0000 & -6.83E-11 & 2.92E-11 & -4.36E-11 & 5.89E-12\\
  &381   & 1.0000 & -1.38E-10 & 7.00E-12 & -1.61E-10 & 5.55E-15\\
  &2193  & 1.0000 & -5.88E-10 & 1.65E-11 & -6.23E-10 & 9.55E-15\\
  &14745 & 1.0000 & -1.10E-08 & -4.27E-09 & -2.00E-08 & 1.73E-14\\
\hline
  \multirow{4}{*}{\rotatebox[origin=c]{90}{$\la{B}_M$}}&87    & 1.0001 & -6.64E-11 & 4.46E-12 & -1.10E-04 & 1.03E-11\\
  &381   & 1.0002 & -1.35E-10 & -1.06E-11 & -2.33E-04 & -5.33E-12\\
  &2193  & 1.0004 & -5.73E-10 & -1.12E-11 & -4.00E-04 & 5.91E-12\\
  &14745 & 1.0005 & -2.37E-09 & -7.73E-11 & -4.97E-04 & -4.47E-11\\
\hline
\hline
  \multirow{4}{*}{\rotatebox[origin=c]{90}{$\la{B}_E$}}&210  & 1.0000 & -3.91E-12 & -4.46E-13 & -4.58E-12 & 9.33E-15\\
  &462  & 1.0000 & -3.82E-12 & -8.91E-13 & -4.55E-12 & 5.77E-15\\
  &1764 & 1.0000 & -9.32E-12 & -4.40E-12 & -1.08E-11 & 1.31E-14\\
  &8292 & 1.0000 & -3.71E-11 & -1.74E-11 & -4.06E-11 & 6.26E-14\\
\hline
  \multirow{4}{*}{\rotatebox[origin=c]{90}{$\la{B}_M$}}&210  & 1.0752 & 1.84E-02 & 7.00E-02 & -7.00E-02 & -2.57E-06\\
  &462  & 1.0219 & 1.94E-03 & 2.14E-02 & -2.14E-02 & -2.21E-06\\
  &1764 & 1.0069 & 1.14E-03 & 6.82E-03 & -6.82E-03 & -4.57E-07\\
  &8292 & 1.0022 & 1.60E-04 & 1.66E-03 & -2.17E-03 & -2.10E-08\\
\hline
\end{tabular}
}
\label{tab:eigenvalues}
\end{center}
\end{table}

\bibliographystyle{siam}
\bibliography{rm_paper}
\end{document}